\documentclass[10pt,a4paper]{amsart} 
\usepackage[english]{babel} 
\usepackage{graphics,color,pgf}
\usepackage{epsfig}
\usepackage[ansinew]{inputenc}
\usepackage[all]{xy}
\usepackage{tikz-cd}
\usetikzlibrary{automata, arrows.meta, positioning}
\usepackage{bbm}
\usetikzlibrary{matrix,arrows,decorations.pathmorphing}
\newdir{ >}{!/8pt/@{}*@{>}}
\usepackage{amssymb, amsmath,amsthm, mathtools, amscd}

\usepackage{mathrsfs}
\usepackage[pagebackref, pdfborder={0 0 0}
  ,urlcolor=black,hypertexnames=false]{hyperref}

\normalfont

\makeatletter
\newtheorem*{rep@theorem}{\rep@title}
\newcommand{\newreptheorem}[2]{%
\newenvironment{rep#1}[1]{%
 \def\rep@title{#2 \ref{##1}}%
 \begin{rep@theorem}}%
 {\end{rep@theorem}}}
\makeatother

\newtheorem{intro_thm}{Theorem}
\newtheorem{intro_cor}[intro_thm]{Corollary}

\newtheorem{intro_defn}[intro_thm]{Definition}
\newtheorem{intro_ques}[intro_thm]{Question}

\newtheorem{lemma}{Lemma}[section]
\newtheorem{thm}[lemma]{Theorem} 
\newtheorem{prop}[lemma]{Proposition}
\newtheorem{cor}[lemma]{Corollary}

\theoremstyle{definition}
\newtheorem{defn}[lemma]{Definition}
\newtheorem{es}[lemma]{Example}
\newtheorem{setup}[lemma]{Setup}

\theoremstyle{remark}
\newtheorem{oss}[lemma]{Remark}

\newtheoremstyle{TheoremNum}
        {0.2 cm}{0.2 cm}              
        {\itshape}                      
        {}                              
        {}                     
        {.}                             
        { }                             
        {\thmname{\bfseries #1}\thmnote{ \bfseries #3}}
    \theoremstyle{TheoremNum}
   
\newtheorem{rec_thm}{Theorem}

\newtheorem{rec_cor}[rec_thm]{Corollary}
\newtheorem{rec_defn}[rec_thm]{Definition}

\newcommand{\id}{\mathrm{id}}

\newcommand{\Hm}{\textup{H}}
\newcommand{\Hb}{\textup{H}_{\textup{b}}}
\newcommand{\Hmb}{\textup{H}_{\textup{mb}}}

\newcommand{\Hn}{\textup{H}_n}
\newcommand{\Hk}{\textup{H}_k}

\newcommand{\Cb}{\textup{C}_{\textup{b}}}

\newcommand{\Cn}{\textup{C}_{n}}
\newcommand{\Cm}{\textup{C}}
\newcommand{\Ck}{\textup{C}_{k}}

\newcommand{\Linf}{\text{L}^{\infty}}
\newcommand{\Ima}{\textup{Im}}
\newcommand{\Id}{\textup{Id}}
\newcommand{\Sing}{\textup{Sing}}
\newcommand{\Linfw}{\textup{L}_{\textup{w}^*}^{\infty}}

\begin{document}

\title[Simplicial volume via foliated simplices and duality]{Simplicial volume via foliated simplices and duality}

\author[F. Sarti]{Filippo Sarti}
\address{Department of Mathematics, University of Pisa, Italy}
\email{filippo.sarti@dm.unipi.it}

\date{\today.}

\begin{abstract}
Let $M$ be a triangulated oriented closed connected manifold with universal cover $\widetilde{M}\to M$ and fundamental group $\Gamma=\pi_1(M)$ and consider an essentially free measure preserving action $\Gamma\curvearrowright (X,\mu)$ on a standard Borel probability space.
We study the space $\Gamma\backslash(\widetilde{M}\times X)$ equipped with the measured foliation defined by Sauer and the theory of singular foliated simplices in this setting.
We define its real singular foliated homology and compare it to classical singular homology. In particular, we construct a foliated fundamental class and prove that its norm coincides with the simplicial volume of $M$, formalizing ideas of Gromov. 
Passing to the dual chain complex, we define the singular foliated bounded cohomology. When $M$ is aspherical we establish an isometric isomorphism with the measurable bounded cohomology of the action groupoid
$\Gamma\curvearrowright X$. As a consequence of a foliated duality principle, we odeduce vanishing criteria for the simplicial volume of $M$ in terms of the vanishing of the measurable bounded cohomology of the action groupoid and/or of its transverse groupoids. 
\end{abstract}
  
\maketitle

\section{Introduction}

The role of dynamics in the interaction between the geometry of manifolds and the algebra of their fundamental group was probably first understood by Gromov and dates back to the 80's. This breakthrough intuition gave birth to still actual research fields going by the name of \emph{geometric group theory} and \emph{measured group theory}. In this paper we follow this flow, being particularly interested in the interplay between a topological invariant called \emph{simplicial volume} and measure preserving (p.m.p) actions of groups on standard Borel probability spaces.
In his book ``\emph{Metric structures for Riemannian and non-Riemannian spaces}" \cite{gromov:libro} Gromov outlined a strategy to study simplicial volume via foliations, hinting a possible definition of simplicial volume for a \emph{measured foliation with transverse measure} (see also\cite[Section 2.4.B]{gromov:fpp}). 
The rough idea, that is attributed to Connes \cite[pg. 303]{gromov:libro}, is first to take a \emph{measure-theoretic} version of singular simplices on the foliation with nice transversality and finiteness properties. Then, one should be able to compare the homology of foliated simplices with the classic singular homology of the leaves and to define a notion of \emph{fundamental class} for the foliation. The norm of the latter would be the simplicial volume of the given foliation.
For example, given a $n$-manifold $M$ with fundamental group $\Gamma=\pi_1(M)$ and universal cover $\widetilde{M}\to M$ together with an essentially free measure preserving action $\Gamma\curvearrowright (X,\mu)$ on a standard Borel probability space, one can consider the measured foliation $\Gamma\backslash(\widetilde{M}\times X)$.
Here, each leaf is a copy of $\widetilde{M}$. 
In this case Gromov claims that the norm of the foliated fundamental class, that we temporarily denote by $[\Gamma\backslash(\widetilde{M}\times X)]$, coincides with the $\ell^1$-norm of the real fundamental class of $M$, namely its simplicial volume $||M||$. He writes that \emph{``...by Connes theorem, we have $|[\Gamma\backslash(\widetilde{M}\times X)]|=||M||$ for all $X$''} \cite[pg. 303]{gromov:libro}. Unfortunately, neither a precise definition of measurable simplex and of fundamental class nor a (reference for a) proof of the above statement is provided. One of the purposes of this paper is exactly to formalize the ideas sketched above and to give both precise definitions and rigorous proofs. Concretely, we first aim to prove the following result, that should correspond to what Gromov called \emph{``Connes theorem"} in \cite[pg. 303]{gromov:libro}. 
\begin{intro_thm}\label{theorem:simplicial:volume}
  Let $M$ be a triangulated oriented closed connected $n$-manifold. Let $p: \widetilde{M}\to M$ be its universal cover and consider an essentially free measure preserving action $\Gamma=\pi_1(M)\curvearrowright (X,\mu)$ on a standard Borel probability space. Then 
  $$||M||= |[\Gamma\backslash(\widetilde{M}\times X)]|\,.$$
\end{intro_thm}
Before sketching the proof, we need to clarify all the terminology appearing in the statement. We start with a brief description of classic notions such as \emph{simplicial volume} and \emph{bounded cohomology}.

\subsection*{Simplicial volume.}
The \emph{simplicial volume} is a topological invariant of manifolds intimately connected with the geometry that they can carry and with the algebraic-dynamical properties of their fundamental group. 
First introduced by Gromov and used in the beautiful geometric proof of Mostow Rigidity \cite{Thurston}, its study became soon an independent research field thanks to the connection with geometric group theory. 
Given an oriented closed connected manifold $M$, its simplicial volume $||M||$ is defined by taking the infimum among the \emph{$\ell^1$-norm} of all representatives of the fundamental class (Definition \ref{def:simp:vol}). Despite the simple definition, a direct computation is generally impossible. Indeed, the exact value of such invariant is known only for few classes of manifolds and most of the significant results rely on geometric or algebraic arguments. This is the case of hyperbolic manifolds, for which it holds a proportionality with the Riemannian volume \cite{loh}, or of manifolds with \emph{amenable} fundamental group, whose simplicial volume vanishes \cite[Corollary 7.12]{miolibro}. 

\subsection*{(Dicrete) bounded cohomology.}
Amenability of groups is an algebraic-dynamical property of groups that can be partially detected considering the dual theory of simplicial volume, namely \emph{bounded cohomology}. The latter was introduced by Johnson and Trauber in the '70s \cite{Johnson} and is the cohomology obtained by taking singular cochains that are bounded with respect to the norm induced by the $\ell^1$-norm on the dual. 
Its connection with geometry and in particular with simplicial volume was first discovered by Gromov. In particular, thanks to the \emph{duality principle} \cite[Proposition 7.10]{miolibro}, the simplicial volume of a manifold $M$ vanishes if and only if the comparison map $\Hb^n(M;\mathbb{R})\to \Hm^n(M;\mathbb{R})$ is zero. Furthermore, as in the unbounded case, the bounded cohomology of $M$ coincides with the one of its fundamental group \cite[Theorem 5.9]{miolibro}. 
As a particular consequence, since the bounded cohomology of amenable groups is zero in all positive degrees, the simplicial volume of a manifold vanishes as soon as its fundamental group is amenable.
It becomes clear that studying simplicial volume is very often related to study bounded cohomology. 
Unfortunately, explicit computations are difficult to obtain also in bounded cohomology.
However, the relation between these invariants showcases how the simplicial volume of a manifold is influenced by the algebraic and the dynamical properties of its fundamental group. 
This fact motivated several authors to approach simplicial volume from a dynamical point of view, for instance by studying a measure-theoretic variant called \emph{integral foliated simplicial volume} \cite{gromov:libro,LP,FLPS}, or by investigating the behaviour with respect to equivalence relations arising in ergodic theory, for example \emph{orbit} and \emph{measure equivalence} \cite{BFS}. 

\subsection*{Measurable bounded cohomology.}
Motivated by role of bounded cohomology in dynamics and by previous works about measurable cocycles and their rigidity, the author and Savini recently introduced a new notion of bounded cohomology for \emph{measured groupoids} \cite{sarti:savini:groupoids}. For example, given an essentially free measure preserving action $\Gamma\curvearrowright (X,\mu)$, the \emph{measurable bounded cohomology} (with trivial coefficients) of the action groupoid $\Gamma\ltimes X$
is obtained from the complex of classes of $\Gamma$-invariant $\mathbb{R}$-valued essentially bounded measurable functions on the product $\Gamma^{\bullet+1}\times X$. The definition comes as a natural straightening of Monod' work about continuous bounded cohomology for topological groups \cite{monod:libro} and in fact the theories are connected by the following \emph{exponential isomorphism} \cite[Theorem 1]{sarti:savini:groupoids}
$$\Hmb^k(\Gamma\ltimes X;\mathbb{R})\cong \Hb^k(\Gamma;\Linf(X,\mathbb{R}))\,$$
that holds for every $k\geq 0$.

\subsection*{Foliated singular homology.}
Going back to foliations and to the statement of Theorem \ref{theorem:simplicial:volume}, we need to fix a suitable definition of \emph{simplices}.
Over the past three decades, measurable analogues of singular simplices have been studied by various authors in different contexts. Notably, Gaboriau developed the theory of \emph{$\mathcal{R}$-simplicial complexes} and their homology to study the behaviour of $\ell^2$-Betti numbers under measure equivalence \cite{Gaboriau2002}. Later on, in his PhD thesis Schmidt \cite{schmidt} introduced the \emph{$\mathcal{R}$-homology} of \emph{$\mathcal{R}$-spaces}, which are fiberwise geometric realizations of $\mathcal{R}$-simplicial complexes.
He further provided a homological characterization and some applications to $\ell^2$-Betti numbers. 
For our purposes, we adopt the approach developed by Sauer, who defined \emph{singular foliated simplices} and employed them as a tool to study the relation between $\ell^2$-Betti numbers and the (minimal) volume of aspherical manifolds \cite{sauer}. 
It is important to emphasize that Sauer's definition (like ours) does not coincide with Schmidt's. The differences between these approaches will be discussed in more detail in Remark \ref{remark:diff}.
To build some intuition, we now focus on the $\mathcal{R}$-space $\Gamma \backslash (\widetilde{M} \times X)$ introduced earlier. Roughly speaking, in this setting a \emph{singular foliated $k$-simplex} is a countable-to-one measurable map from a subset $A$ of the reals of finite Lebesgue measure to the quotient $\Gamma\backslash(\Sing_k(\widetilde{M})\times X)$ such that: (i) the induced \emph{natural measure} coincides with the Lebesgue measure of $A$; (ii) the chain defined on each leaf is locally finite (Definition \ref{definition:singular:simplex}). We warn the reader that the precise meaning of words like ``measurable" and ``natural measure" requires a careful and systematic study that we postpone to Section \ref{section:foliated}.
With such simplices one can build (integral or real) chains and, after discarding ``null chains" in order to avoid non-zero objects with zero norm, define the \emph{singular foliated homology} of the foliation $\Gamma \backslash(\widetilde{M}\times X)$. 
For our scopes, in this paper we replace the integral coefficients used by Sauer \cite{sauer} with real coefficients, thus we consider the \emph{real singular foliated homology} $\Hk(\Gamma \backslash(\widetilde{M}\times X);\mathbb{R})$ (Definition \ref{definition:homology}). 
Furthermore, instead of taking the \emph{mass} of simplices (which takes into account the measure of the support of their image), we consider the \emph{weight}, namely the support of their domain (Definition \ref{definition:singular:simplex}). This turns the chain complex of singular foliated chains into a normed chain complex.

In order to recover the usual notion of simplicial volume, the following task is to find a suitable notion of \emph{fundamental class} in this setting.
The idea of Sauer \cite[Lemma 3.23]{sauer} was to construct a map 
$$\Hn(M;\Linf(X,\mathbb{Z}))\to \Hn(\Gamma \backslash(\widetilde{M}\times X))$$
and to take the image under such map of the \emph{$X$-fundamental class} $[M]_{\mathbb{Z}}^X\in \Hn(M;\Linf(X,\mathbb{Z}))$ (Definition \ref{def:fund:classes}). However, when the coefficients are reals, his construction does not adapt straightforwardly and needs a slight refinement. 
More precisely, we replace the space $\Cn(\widetilde{M})\otimes_{\mathbb{Z}\Gamma} \Linf(X,\mathbb{Z})$ with it ``realification''
$\Cn(\widetilde{M})\otimes_{\mathbb{Z}\Gamma} \Linf(X,\mathbb{Z})\otimes_{\mathbb{Z}\Gamma} \mathbb{R}$. The latter carries a natural injection of both real singular chains and integral parametrized chains, so that the following diagram made by inclusions commutes
\begin{center}
  \begin{tikzcd}
  \Cn(M)\arrow{r}{i}\arrow{d}[swap]{j}&\Cn(\widetilde{M})\otimes_{\mathbb{Z}\Gamma} \Linf(X,\mathbb{Z})\arrow{d}{\ell}\\
  \Cn(M;\mathbb{R})\arrow{r}{k}&\Cn(\widetilde{M})\otimes_{\mathbb{Z}\Gamma}\Linf(X,\mathbb{Z})\otimes_{\mathbb{Z}\Gamma} \mathbb{R}\,.
  \end{tikzcd}
\end{center}
The \emph{real $X$-fundamental class} $[M]^X_{\mathbb{R}}\in \Hn(M;\Linf(X,\mathbb{Z})\otimes_{\mathbb{Z}\Gamma}\mathbb{R})$ is the image of the real fundamental class under the map induced in homology by the bottom horizontal map (or, equivalently, of the $X$-fundamental class under the map induced in homology by the right vertical map). 
We are now able to adapt Sauer's construction, namely to build a map (Lemma \ref{lemma:lambda})
$$\Lambda_n:\Hn(M;\Linf(X,\mathbb{Z})\otimes_{\mathbb{Z}\Gamma} \mathbb{R})\to \Hn(\Gamma \backslash(\widetilde{M}\times X);\mathbb{R})\,.$$
The \emph{$X$-foliated fundamental class} $[\Gamma\backslash(\widetilde{M}\times X)]$ is the image of the real $X$-fundamental class under $\Lambda_n$ (Definition \ref{def:fund}). 

Once all the ingredients are in place we can move to the proof of Theorem \ref{theorem:simplicial:volume}, which requires two steps. The first one is to show that the norm of the real $X$-fundamental class coincides with the simplicial volume of $M$, that is the norm of its real fundamental class. This is an easy computation (Lemma \ref{lemma:simplicial:volume}).
The second step is more challenging and consists to show that $\Lambda_n$ is isometric. Here we use arguments due to Sauer to approximate foliated chains with chains in $\Cn(\widetilde{M})\otimes_{\mathbb{Z}\Gamma} \Linf(X,\mathbb{Z}) \otimes_{\mathbb{Z}\Gamma} \mathbb{R}$ (Lemma \ref{lemma:lambda}).

\medskip

\subsection*{Singular foliated bounded cohomology.}
In the second part of the paper we reverse our perspective. Motivated by the duality between simplicial volume and bounded cohomology, we consider the topological dual of singular foliated chains
$$(\Cb^{\bullet}(\Gamma \backslash(\widetilde{M}\times X); \mathbb{R}),\delta^{\bullet})\coloneqq(    \Cm_{\bullet}(\Gamma\backslash(\widetilde{M}\times X);\mathbb{R}),\partial_{\bullet})'\,.$$
\begin{intro_defn}\label{def:foliated:bc}
  The \emph{singular foliated bounded cohomology} of $M$ is 
  $$\Hb^k(\Gamma \backslash(\widetilde{M}\times X); \mathbb{R})\coloneqq \Hm^k(
      \Cb^{\bullet}(\Gamma \backslash(\widetilde{M}\times X); \mathbb{R}),\delta^{\bullet})\,.$$
\end{intro_defn}

Our last goal is to relate the above cohomology with the ordinary bounded cohomology of $M$ (or, equivalently, of its fundamental group). 
To this end we need to introduce a non-equivariant version of singular foliated simplices (Definition \ref{definition:singular:simplex:non:eq}) and to show that the associated cochain complex $(\Cb^{\bullet}(\widetilde{M}\times X; \mathbb{R}),\delta^{\bullet})$ satisfies the hypothesis of the Fundamental Lemma of Homological Algebra for bounded cohomology (Lemma \ref{lemma:fundamental}). This approach presents a technical difficulty that we overcome by passing to the \emph{completion} of the involved normed chain complexes, in order to admit infinite sums as representatives for chains (refer to Section \ref{sec:completion} for all the details). 
Then, we show that each new module is a \emph{relatively injective} $\Gamma$-module. Moreover, we add to the complex an augmentation map $\Linf(X,\mathbb{R})\to \Cb^0(\widetilde{M}\times X;\mathbb{R})'$, turning it into a \emph{strong} resolution of the $\Gamma$-module $\Linf(X,\mathbb{R})$.
As a consequence we deduce the following
\begin{intro_thm}\label{theorem:bounded:cohomology}
  Let $M$ be a triangulated oriented closed connected aspherical manifold. Let $p: \widetilde{M}\to M$ be its universal cover and consider an essentially free measure preserving action $\Gamma=\pi_1(M)\curvearrowright (X,\mu)$ on a standard Borel probability space. Then we have canonical isometric isomorphisms 
  $$\Hb^k(\Gamma \backslash(\widetilde{M}\times X); V)\cong \Hb^k(\Gamma;\Linf(X,\mathbb{R}))\cong \Hmb^k(\Gamma\ltimes X;\mathbb{R})$$
  for every $k\geq 0$. 
\end{intro_thm}

As an application of Theorem \ref{theorem:bounded:cohomology} we obtain a different proof of the following criterion for the vanishing of simplicial volume. This was implicitly contained in the work by Bader, Furman and Sauer \cite{BFS} as a consequence of duality with (non-trivial) coefficients \cite[Proposition 2.3.1]{monod:libro}.

\begin{intro_cor}\label{corollary:vanishing}
  Let $M$ be a triangulated oriented closed connected aspherical $n$-manifold. Let $p: \widetilde{M}\to M$ be its universal cover and consider an essentially free measure preserving action $\Gamma=\pi_1(M)\curvearrowright (X,\mu)$ on a standard Borel probability space. If $\Hmb^n(\Gamma\ltimes X;\mathbb{R})$ (equivalently $\Hb^n(\Gamma;\Linf(X,\mathbb{R}))\cong \Hb^n(M;\Linf(X,\mathbb{R}))$) vanishes, then also $||M||$ does. 
\end{intro_cor}

We conclude the paper with an application to \emph{transverse measured groupoids}, which combines Corollary \ref{corollary:vanishing} with the theory of bounded cohomology of groupoids due to the author and Savini \cite{sarti:savini:groupoids} and recent results by Hartnick and the author about cohomological induction \cite{HS}.
Given a p.m.p action $\Gamma\curvearrowright (X,\mu)$, a \emph{cross-section} is a measurable subset $Y\subset X$ meeting any orbit (Definition \ref{def:cross}). The associated \emph{transverse measured groupoid} is the restriction $\Gamma\ltimes X|_Y$, equipped with the measure structure described for instance by Bj\"orklund, Hartnick and Karasik in \cite{BHK}. Thanks to the isomorphism $\Hmb^n(\Gamma\ltimes X|_Y;\mathbb{R})\cong \Hmb^n(\Gamma\ltimes X;\mathbb{R})$ proved by Hartnick and the author \cite{HS}, we obtain the following result. 

\begin{intro_cor}\label{corollary:transverse}
  Let $M$ be a triangulated oriented closed connected aspherical $n$-manifold. Let $p: \widetilde{M}\to M$ be its universal cover and consider an essentially free measure preserving action $\Gamma=\pi_1(M)\curvearrowright (X,\mu)$ on a standard Borel probability space. Let $Y\subset X$ be any cross-section. If $\Hmb^n(\Gamma\ltimes X|_Y;\mathbb{R})$ vanishes, then also $||M||$ does. 
\end{intro_cor}

On one hand, Corollary \ref{corollary:vanishing} and Corollary \ref{corollary:transverse} provide a new way to produce examples of manifolds with vanishing simplicial volume. One the other one, the lack of concrete computations in bounded cohomology damps the enthusiasm about their effective impact and rises the following 
\begin{intro_ques}\label{question}
  Let $M$ be a triangulated oriented closed connected aspherical $n$-manifold whose fundamental group $\Gamma$ contains an infinite amenable normal subgroup. Is there any p.m.p action $\Gamma\curvearrowright (X,\mu)$ such that $\Hmb^n(\Gamma\ltimes X;\mathbb{R})=0$?
  Is there any such action admitting a transverse measured groupoid with vanishing bounded cohomology, at least in degree $n$?
  \end{intro_ques}
We point out that a positive answer to the above question would provide new examples of manifolds satisfying a question posed by L\"uck \cite[Question 14.39]{luck}.

\vspace{5pt} 
\paragraph{\textbf{Structure of the paper.}}
The paper is divided in five sections. In Section \ref{sec:prel} we recall the basics about bounded cohomology of groups, topological spaces and measured groupoids (Section \ref{sec:bc:prel}) and their homological algebraic characterization (Section \ref{sec:homo}). In Section \ref{sec:fund} we focus on the notion(s) of fundamental class(es) and how they are related to simplicial volume. Then we move to Section \ref{section:R-category}, where we describe the categories of $X$-spaces (Section \ref{sec:X:spaces}) and $X$-maps and of $\mathcal{R}$-spaces and geometric maps (Section \ref{sec:R:spaces}). Then, in Section \ref{sec:foliated}, we define foliated singular simplices and their homology and we conclude with Section \ref{sec:funct} proving some functorial properties.
In Section \ref{section:foliated} we apply this theory for the $\mathcal{R}$-space $\widetilde{M}\times X$, giving the proof of Theorem \ref{theorem:simplicial:volume}. Finally, in Section \ref{section:foliated:cohomology}, we consider the singular foliated bounded cohomology, we prove Theorem \ref{theorem:bounded:cohomology} and its corollaries.
\vspace{5pt}

\paragraph{\textbf{Acknowledgements.}} 
The author is indebted to Federica Bertolotti, Michelle Bucher, Pietro Capovilla, Stefano Francaviglia, Thorben Kastenholz, Roman Sauer and Alessio Savini for valuable discussions and advices happened in different moments during the drafting of this paper. 
I am truly grateful to Marco Moraschini, for his patience in reading a preliminary draft of this paper and for his precious comments.
Finally, I express my gratitude to the anonymous referee, who pointed out some weaknesses in a previous version of this article and allowed me to improve its quality.
The author was partially supported by INdAM through GNSAGA, and by MUR through the PRIN project ``Geometry and topology of manifolds".


\section{Preliminaries}\label{sec:prel}
For the material described in this section the reader should consider standard references as \cite{Ivanov,monod:libro,miolibro,sarti:savini:groupoids,sarti:savini25}.

\subsection{Bounded cohomology of groups, spaces and actions.}\label{sec:bc:prel}
Let $\Gamma$ be a discrete countable group and consider a normed $\Gamma$-module, that is a normed vector space $V$ with an isometric $\Gamma$-action
$$\Gamma\times V\to V\,,\quad (\gamma,v)\mapsto \gamma\cdot v\,.$$
We set 
$$\Cm^k(\Gamma;V)\coloneqq \{f:\Gamma^{k+1}\to V\}\,$$
and define 
$$\delta^k:\Cm^k(\Gamma;V)\to \Cm^{k+1}(\Gamma;V)$$
as $$ \delta^k f(\gamma_0,\ldots,\gamma_{k+1})\coloneqq \sum\limits_{i=0}^{k+1} (-1)^i f(\gamma_0,\ldots,\widehat{\gamma_i},\ldots,\gamma_{k+1})\,.$$
It is easy to check that $\delta^{k+1}\circ \delta^k=0$ and $||\delta^k||_{\text{op}}\leq k+1$ where $||\cdot||_{\text{op}}$ is the operator norm. 
Thus the sets
$$\Cb^k(\Gamma;V)\coloneqq \left\{f\in \Cm^k(\Gamma;V)\,,\, \sup_{\gamma_0,\ldots,\gamma_{k}} ||f(\gamma_0,\ldots,\gamma_{k})||<+\infty\right\}\,.$$
form a cochain complex $(\Cb^\bullet(\Gamma;V),\delta^{\bullet})$. 
Moreover, each $\Cb^k(\Gamma;V)$ is a normed $\Gamma$-module with the action 
$$(\gamma\cdot f)(\gamma_0,\ldots,\gamma_{k})\coloneqq \gamma\cdot f(\gamma^{-1}\gamma_0,\ldots,\gamma^{-1}\gamma_{k})$$
and the supremum norm $||\cdot||_{\infty}$.
\begin{defn}
  The \emph{bounded cohomology} of $\Gamma$ with coefficients in $V$ is the cohomology of the subcomplex of $\Gamma$-invariants $(\Cb^\bullet(\Gamma;V)^\Gamma,\delta^{\bullet})$,
  namely $$\Hb^k(\Gamma, V)\coloneqq \Hm^k(\Cb^\bullet(\Gamma;V)^{\Gamma},\delta^{\bullet})\,.$$
  The norm $||\cdot||_{\infty}$ on $\Cb^k(\Gamma;V)$ descends to a \emph{semi-norm} on $\Hb^k(\Gamma, V)$, still denoted $||\cdot||_{\infty}$.
\end{defn}

Let $Y$ be a topological space with fundamental group $\Gamma=\pi_1(Y)$ and $(\Cm^\bullet(Y;V),\delta^{\bullet})$ be its complex of singular cochains with coefficients in $V$.
Define
$$\Cb^{k}(Y;V)\coloneqq \left\{f\in \Cm^{k}(Y;V)\,,\, \sup_{s\in \Sing_{k}(Y)} ||f(s)||<+\infty\right\}\,,$$
where $\Sing_k(Y)$ is the set of singular $k$-simplices of $Y$. Then $||\delta^{n}||_{\text{op}}\leq k+1$, whence $\delta^{k}$ restricts to bounded cochains. 
\begin{defn}
  The \emph{bounded cohomology} of $Y$ with coefficients in $V$ is the cohomology of the complex
  $(\Cb^\bullet(Y;V),\delta^{\bullet})$,
  namely $$\Hb^k(Y, V)\coloneqq \Hm^k(\Cb^\bullet(Y;V),\delta^{\bullet})\,.$$
  The operator norm on each $\Cb^{k}(Y;V)$ descends to a semi-norm, denoted again $||\cdot||_{\infty}$.
\end{defn}

\medskip
Assume now that $Y$ is also path connected and admits a universal cover $p: \widetilde{Y}\to Y$. 
As above, we consider the cochain complex
$(\Cb^{\bullet}(\widetilde{Y};V),\delta^{\bullet})$ of singular cochains in $\widetilde{Y}$, which is endowed with the $\Gamma$-action inherited from the deck transformations of $p:\widetilde{Y}\to Y$. Then we have a natural isometric identification
$$\Cb^{k}(Y;V)\cong \Cb^{k}(\widetilde{Y};V)^{\Gamma}$$
in every degree
that descends in homology to a canonical isometric isomorphism
\begin{equation}\label{equation:isom:iso0}
  \Hb^k(Y;V)\cong \Hm^k(\Cb^{\bullet}(\widetilde{Y};V)^{\Gamma},\delta^{\bullet})\,.
\end{equation}
Such obvious identification is the base step to connect the bounded cohomology of $Y$ with the one of its fundamental group. 
\begin{thm}[Gromov, Ivanov]
Let $Y$ be a path connected topological space with fundamental group $\Gamma$ and universal cover $p:\widetilde{Y}\to Y$. Then we have isometric isomorphisms
\begin{equation}\label{equation:mapping:thm}
\Hb^{k}(\Gamma;\mathbb{R})\cong \Hb^{k}(Y;\mathbb{R})
\end{equation} for every $k\geq 0$. 

If $Y$ is moreover aspherical, then the above isometric isomorphism holds for every Banach $\Gamma$-module $V$, namely
\begin{equation}\label{equation:mapping:thm:coefficients}
  \Hb^{k}(\Gamma;V)\cong \Hb^{k}(Y;V)
  \end{equation} for every $k\geq 0$.
\end{thm}

Let $\Gamma\curvearrowright (X,\mu)$ be an essentially free measure preserving action on a standard Borel probability space. Notice that, for any infinite countable group examples of such actions always exist. For instance, the action of $\Gamma$ on its \emph{Bernoulli shift} $B_\Gamma$ and, in the residually finite case, on its \emph{profinite completion} $\widehat{\Gamma}$ \cite{FLPS}. Consider a Banach $\Gamma$-module $V$ which is the dual of a separable Banach $\Gamma$-module $W$. 
Consider the set 
$$\Linfw(\Gamma^{k+1}\times X,V)$$ 
of classes of $V$-valued weak$^*$-measurable essentially bounded functions on $\Gamma^{k+1}\times X$
endowed with the essentially supremum norm and with the $\Gamma$-action defined on representatives as
$$(\gamma \cdot f) (\gamma_0,\ldots,\gamma_k,x)\coloneqq \gamma \cdot f(\gamma^{-1}\gamma_0,\ldots,\gamma^{-1}\gamma_k,\gamma^{-1}x)\,.$$
We point out that, in general, weak$^*$-measurability does not coincide with measurability with respect to the topology induced by the norm on $V$. 
The coboundary operator defined above extend to 
$$\delta^k:\Linfw(\Gamma^{k+1}\times X,V)\to \Linfw(\Gamma^{k+2}\times X,V)\,.$$
as follows: for any $F\in \Linfw(\Gamma^{k+1}\times X,V)$ and any representative $f$, we define $\delta^k F$ as the class of $$ \delta^k f(\gamma_0,\ldots,\gamma_{k+1},x)\coloneqq \sum\limits_{i=1}^{k+1} (-1)^i f(\gamma_0,\ldots,\widehat{\gamma_i},\ldots,\gamma_{k+1},x)\,.$$
\begin{defn}
  The \emph{measurable bounded cohomology} of the action groupoid $\Gamma\ltimes X$ with coefficients in $V$ is the cohomology of the complex
  $(\Linfw(\Gamma^{\bullet+1}\times X,V)^{\Gamma},\delta^{\bullet})$,
  namely $$\Hmb^k(\Gamma\ltimes X; V)\coloneqq \Hm^k(\Linfw(\Gamma^{\bullet+1}\times X,V)^{\Gamma},\delta^{\bullet})\,.$$
\end{defn}
As a consequence of the exponential isomorphism \cite[Corollary 2.3.3]{monod:libro}
$$\Linfw(\Gamma^{k+1}\times X,V)\cong \Linfw(\Gamma^{k+1},\Linfw(X,V))$$
we get a canonical isometric isomorphism
\begin{equation}\label{equation:exponential:isomorphism}
  \Hmb^k(\Gamma\ltimes X;V)\cong \Hb^k(\Gamma;\Linfw(X,V)) 
\end{equation}
for every $k\geq 0$ \cite[Theorem 1]{sarti:savini:groupoids}.

The following general duality principle \cite[Lemma 6.1]{miolibro} will be useful later on. 
\begin{lemma}\label{lemma:duality}
Let $(C,\partial,||\cdot||_1)$ be a normed chain complex with dual normed chain complex $(C',\delta,||\cdot||_{\infty})$. Fix $n\in \mathbb{N}$ and denote by $\langle\cdot,\cdot\rangle$ the Kronecker pairing between $\Hm_n(C)$ and $\Hb^n(C')$. Then for every $\alpha\in \Hm_n(C)$ we have
$$||\alpha||_1 = \max\left\{\langle\beta,\alpha\rangle\,|\,\beta \in \Hb^n(C')\,,\, ||\beta||_{\infty} \leq 1\right\}\,.$$
\end{lemma}

\subsection{Homological algebra.}\label{sec:homo}
Since we work in the category of normed $\Gamma$-modules, all the maps will be assumed to be \emph{$\mathbb{R}$-linear} and \emph{bounded}. 
For such a map $\varphi:V\to W$ we recall that 
$$||\varphi||=\sup\limits_{v\in V} \frac{||\varphi(v)||_W}{||v||_V}\,.$$

A map $\iota: E \to W$ between normed $\Gamma$-modules is \emph{strongly injective} if there
is a map $\varphi: W \to E$ with $||\varphi||\leq 1$ such that $\varphi\circ \iota = \Id_E$ (in particular, $\iota$ is
injective).

A normed $\Gamma$-module $V$ is \emph{relatively injective}
if
the following holds: whenever $E,W$ are normed $\Gamma$-modules, $\iota: E \to W$ is a strongly
injective $\Gamma$-map and $\psi: E\to V$ is a $\Gamma$-map, there exists a $\Gamma$-map $\beta: W \to V$ with $||\beta||\leq ||\psi||$ 
such that $\beta\circ \iota =\psi$, namely the following extension problem admits a solution
\begin{center}
  \begin{tikzcd}
E \arrow{rr}{\iota}\arrow[swap]{rd}{\psi} && W\arrow[bend right =20,swap]{ll}{\varphi}\arrow[dotted]{ld}{\beta}\\
& V\,.&
  \end{tikzcd}    
\end{center}

A \emph{complex} of normed $\Gamma$-modules is a cochain complex $(V^{\bullet},\delta^{\bullet})$ where each $V^{\bullet}$ is a normed $\Gamma$-module and each $\delta^{\bullet}$ is a $\Gamma$-map. 

Given a normed $\Gamma$-module $V$, a \emph{resolution} $(V^{\bullet},\delta^{\bullet},\epsilon)$ of $V$ is a complex $(V^{\bullet},\delta^{\bullet})$ where 
$V^n=0$ if $n<0$ together with an augmentation $\Gamma$-map $\epsilon: V\to V^0$ such that the complex 
$$0\longrightarrow V\overset{\epsilon}{\longrightarrow}
V^0\overset{\delta^0}{\longrightarrow}V^1\overset{\delta^1}{\longrightarrow} V^2\overset{\delta^2}{\longrightarrow}\cdots$$ is exact. 

\begin{es}
  The complex $(\Cb^\bullet(\Gamma,V),\delta^\bullet)$ is a resolution of $V$ with the augmentation map given by inclusion of coefficients, namely
  $$\varepsilon:V \to \Cb^0(\Gamma;V)\,,\quad \varepsilon(v)(\gamma)\coloneqq v $$
  for every $v\in V$ and $\gamma\in \Gamma$.
\end{es}

A resolution $(V^{\bullet},\delta^{\bullet},\epsilon)$ of a normed $\Gamma$-module $V$ is \emph{strong} if it admits a contracting homotopy, that is a family of maps 
$k^\bullet: V^\bullet\to V^{\bullet-1}$ and 
$k^0:V^0\to V$
such that 
$$\delta^{\bullet-1} \circ k^{\bullet} + k^{\bullet+1}  \circ \delta^{\bullet} = \Id_{V^\bullet}$$ if $\bullet\geq 0$
and $$k^0 \circ \epsilon= \Id_V\,.$$

\begin{lemma}[Fundamental Lemma of Homological Algebra for Bounded Cohomology]\label{lemma:fundamental}
Let $V$ be a normed $\Gamma$-module. Let 
$(V^{\bullet},\delta_V^{\bullet},\epsilon)$ be a strong resolution of $V$ by relatively injective normed $\Gamma$-modules.
Then we have canonical isomorphisms
$$\Hb^{k}(\Gamma;\mathbb{R})\cong \Hm^{k}((V^{\bullet},\delta_V^{\bullet},\epsilon)^{\Gamma})$$
for every $k\geq 0$. 

If moreover there exists a norm non-increasing chain map 
$(\Cb^\bullet(\Gamma,V),\delta^\bullet,\varepsilon)\to( V^{\bullet},\delta_V^{\bullet},\epsilon)$ extending the identity on $V$, then the above isomorphisms are also isometric. 
\end{lemma}

\section{Fundamental classes and simplicial volume}\label{sec:fund}
In this section we start from the classic notion of integral fundamental class for an oriented closed connected manifold and we construct some variations. Although some of them are classical and have already been studied \cite{LP,FLPS}, we introduce some novelties.

We fix, once and for all, the following
\begin{setup}\label{setup}
  Let $M$ be an oriented closed connected $n$-manifold and let $p:\widetilde{M}\to M$ be its universal cover. With a slight abuse of notation, we write $p$ also for the induced maps $\Sing_k(\widetilde{M})\to \Sing_k(M)$, where $\Sing_k(-)$ is the set of singular $k$-simplices of $-$. Denote by $\Gamma=\pi_1(M)$ and by $D$ a fundamental domain of $\Gamma\curvearrowright \widetilde{M}$. 
For any $k\geq 0$, the induced action $\Gamma\curvearrowright \Sing_k(\widetilde{M})$ inherits the fundamental domain $D_k$ obtained by taking simplices with first vertex in $D$. 
  Let $(X,\mu)$ be a standard Borel probability space and $\Gamma\curvearrowright (X,\mu)$ be an essentially free measure preserving action. 
\end{setup}

Let $$\Ck(\widetilde{M})\coloneqq \mathbb{Z}[\Sing_k(\widetilde{M})]$$ be the complex of integral singular chain on $\widetilde{M}$, endowed with the obvious $\Gamma$-action. We consider the following chain complexes, whose boundary homomorphisms are all defined as extensions of the singular boundary operator:
\begin{itemize}
  \item[(i)] the complex of \emph{integral singular chains} on $M$, that is 
  $$\Ck(M)\coloneqq\mathbb{Z}[\Sing_k(M)]\cong\Ck(\widetilde{M})\otimes_{\mathbb{Z}\Gamma} \mathbb{Z}\,,$$
  where $\Gamma$ acts trivially on $\mathbb{Z}$ and the isomorphism is canonical.
  \item[(ii)] the complex of \emph{real singular chains} on $M$
  $$\Ck(M;\mathbb{R})\coloneqq \mathbb{R}[\Sing_k(M)]   \cong\Ck(\widetilde{M})\otimes_{\mathbb{Z}\Gamma} \mathbb{R}\,,$$
  where $\Gamma$ acts trivially on $\mathbb{R}$ and the isomorphism is canonical.
  \item[(iii)] the complex of \emph{integral parametrized singular chains} on $M$
  $$\Ck(M;\Linf(X,\mathbb{Z}))\coloneqq \Ck(\widetilde{M})\otimes_{\mathbb{Z}\Gamma} \Linf(X,\mathbb{Z})\,,$$
  where $\Gamma$ acts on $\Linf(X,\mathbb{Z})$ as $\gamma\cdot F=[\gamma \cdot f]$ and
  $\gamma\cdot f(x)=f(\gamma^{-1}x)$ where $F\in \Linf(X,\mathbb{Z})$ and $f$ represents $F$.
  \item[(iv)] the complex of \emph{integral parametrized real singular chains} on $M$
  $$\Ck(M;\Linf(X,\mathbb{Z})\otimes_{\mathbb{Z}\Gamma} \mathbb{R})\coloneqq \Ck(\widetilde{M})\otimes_{\mathbb{Z}\Gamma} \Linf(X,\mathbb{Z})\otimes_{\mathbb{Z}\Gamma} \mathbb{R}\,,$$
  where the $\Gamma$-actions on $\mathbb{R}$ and $\Linf(X,\mathbb{Z})$ are as above. 
\end{itemize}

\begin{oss}
By definition, elements of the above tensor products are equivalence classes. A representative $\sum_{i} s_i\otimes -$ of a given class is said to be in \emph{reduced form} if $p(s_i)\neq p(s_j)$ whenever $i\neq j$ \cite[Definition 4.2]{LP}. One can easily show that each class admits such a representative. For this reason, from now on we will tacitly assume that a (representative of a) chain is in reduced form.
\end{oss}

Each of the above modules is in fact a \emph{normed $\Gamma$-module}. Precisely, they can be equipped with a norm as follows. 
\begin{itemize}
  \item[(i)] $$\left|\left|\sum\limits_{i=1}^m a_i s_i\right|\right|_1= \sum\limits_{i=1}^m |a_i|$$
  where $a_i\in \mathbb{Z}$ and $s_i\in \Ck(\widetilde{M})$ with $p(s_i)\neq p(s_j)$ if $i\neq j$. This is called the \emph{$\ell^1$-norm}.
  \item[(ii)] $$\left|\left|\sum\limits_{i=1}^m a_i s_i\right|\right|_1= \sum\limits_{i=1}^m |a_i|$$
  where $a_i\in \mathbb{R}$ and $s_i\in \Ck(\widetilde{M})$ with $p(s_i)\neq p(s_j)$ if $i\neq j$. This is again called \emph{$\ell^1$-norm}.
  \item[(iii)] $$\left\bracevert\sum\limits_{i=1}^m s_i\otimes [f_i]\right\bracevert^X= \sum\limits_{i=1}^m \int_X|f_i(x)|d\mu(x)$$
  where $[f_i]\in \Linf(X,\mathbb{Z})$ and $s_i\in \Ck(\widetilde{M})$ with $p(s_i)\neq p(s_j)$ if $i\neq j$.
  \item[(iv)] $$\left\bracevert\sum\limits_{i=1}^m s_i\otimes [f_i]\otimes a_i \right\bracevert_\mathbb{R}^X= \sum\limits_{i=1}^m\left( |a_i|\int_X|f_i(x)|d\mu(x)\right)$$
  where $a_i\in \mathbb{R},\;[f_i]\in \Linf(X,\mathbb{Z})$ and $s_i\in \Ck(\widetilde{M})$ with $p(s_i)\neq p(s_j)$ if $i\neq j$.
\end{itemize}

The boundary operator of $\Ck(\widetilde{M})$ induces boundary operator on each of the modules described above. Thus, we obtain complexes of normed $\Gamma$-modules, and each norm descends to a seminorm in homology, that we denote with the same symbols. Moreover, we have the following commutative square where all maps are induced by inclusions of coefficients
\begin{equation}\label{diagram}
  \begin{tikzcd}
     \Ck(M)\arrow{r}{i}\arrow{d}[swap]{j}&\Ck(M;\Linf(X,\mathbb{Z}))\arrow{d}{\ell}\\
    \Ck(M;\mathbb{R})\arrow{r}{k}&\Ck(M;\Linf(X,\mathbb{Z})\otimes_{\mathbb{Z}\Gamma} \mathbb{R})\,.
    \end{tikzcd}
\end{equation}
Since all the maps commutes with the boundary operators, the above diagram is made by chain maps and descends in homology to
\begin{center}
  \begin{tikzcd}
  \Hk(M)\arrow{r}{I}\arrow{d}[swap]{J}&\Hk(M;\Linf(X,\mathbb{Z}))\arrow{d}{L}\\
  \Hk(M;\mathbb{R})\arrow{r}{K}&\Hk(M;\Linf(X,\mathbb{Z})\otimes_{\mathbb{Z}\Gamma} \mathbb{R})
  \end{tikzcd}
\end{center}

Furthermore, all the inclusions $\mathbb{Z}\hookrightarrow \mathbb{R}$, $\mathbb{Z}\hookrightarrow \Linf(X,\mathbb{Z})$ and 
$\mathbb{R}\hookrightarrow\Linf(X,\mathbb{Z})\otimes_{\mathbb{Z}\Gamma} \mathbb{R}$ do not increase the norm (as morphisms of normed $\mathbb{Z}\Gamma$-modules). Thus the maps $I,J,K,L$ are norm non-increasing.

Some variations of the notion of \emph{fundamental class} can be defined starting from the integral fundamental class.
\begin{defn}\label{def:fund:classes}
  Denote by $[M]\in \Hn(M)$ the (integral) fundamental class of $M$.

The \emph{real fundamental class} of $M$ is $$[M]_{\mathbb{R}}\coloneqq J([M])\in \Hn(M;\mathbb{R})\,.$$ A \emph{real fundamental cycle} of $M$ is any representative of $[M]_{\mathbb{R}}$ in $\Cn(M;\mathbb{R})$. 

The \emph{$X$-fundamental class} of $M$ is $$[M]^X_{\mathbb{Z}}\coloneqq I([M])\in \Hn(M;\Linf(X, \mathbb{Z}))\,.$$ A \emph{$X$-fundamental cycle} of $M$ is a representative of $[M]^X_{\mathbb{Z}}$ in $\Cn(M;\Linf(X,\mathbb{Z}))$.

The \emph{real $X$-fundamental class} of $M$ is $$[M]^X_{\mathbb{R}}\coloneqq K([M]_{\mathbb{R}})=L([M]^X_{\mathbb{Z}})\in \Hn(M;\Linf(X,\mathbb{Z})\otimes_{\mathbb{Z}\Gamma}\mathbb{R})\,.$$ A \emph{real $X$-fundamental cycle} of $M$ is a representative of $[M]^X_{\mathbb{R}}$ in $\Cn(M;\Linf(X,\mathbb{Z})\otimes_{\mathbb{Z}\Gamma} \mathbb{R})$.
\end{defn}

The following inequalities hold since all the maps in the above diagrams do not increase the norm
\begin{equation}\label{equation:inequalities}
\left\bracevert[M]^X_{\mathbb{R}}\right\bracevert_{\mathbb{R}}^X\leq \left|\left|[M]_{\mathbb{R}}\right|\right|_1\leq \left|\left|[M]\right|\right|_1\,,\;\;\;
\left\bracevert[M]^X_{\mathbb{R}}\right\bracevert^X_{\mathbb{R}}\leq \left\bracevert[M]_{\mathbb{Z}}^X\right\bracevert^X\leq\left|\left|[M]\right|\right|_1 \,.
\end{equation}

Our next goal is to investigate further the relation between the (semi)norms introduced above. In particular, the $\ell^1$-norm of the real fundamental class reflects interesting features of the manifold itself. This led Gromov \cite{vbc} to introduce the following 
\begin{defn}[Gromov]\label{def:simp:vol}
  The \emph{simplicial volume} of $M$ is 
  $$||M||\coloneqq \left|\left|[M]_{\mathbb{R}}\right|\right|_1\,.$$
\end{defn}

Among the inequalities of Equation \eqref{equation:inequalities},
it is not difficult to prove that one is in fact an equality. As a consequence, we obtain the following
\begin{lemma}\label{lemma:simplicial:volume}
  Retain the Setup \ref{setup}. Then
  $$||M||=\left\bracevert[M]^X_{\mathbb{R}}\right\bracevert^X_{\mathbb{R}}\,.$$
\end{lemma}
\begin{proof}
We first show that $K$ admits a left inverse of norm at most one in every degree. Fix $k\geq 0$.
Recall that $K:\Hk(M;\mathbb{R})\to \Hk(M;\Linf(X,\mathbb{Z})\otimes_{\mathbb{Z}\Gamma} \mathbb{R})$ does not increase norm. Furthermore, the integration map $\Ck(M;\Linf(X,\mathbb{Z})\otimes_{\mathbb{Z}\Gamma} \mathbb{R})\to \Ck(M;\mathbb{R})$ (compare also \cite[Remark 4.7]{LP}) defined as 
$$\sum\limits_{i=1}^m s_i\otimes [f_i]\otimes a_i \mapsto \sum\limits_{i=1}^m s_i \otimes a_i \int_X f_i(x) d\mu(x)$$ sends real $X$-fundamental cycles to real fundamental cycles and does not increase the norm. Hence it descends in homology to a left inverse for $K$ of norm at most one.
The conclusion now follows by taking $k=n$ and by the inequality
$$\left\bracevert[M]^X_{\mathbb{R}}\right\bracevert_{\mathbb{R}}^X\leq \left|\left|[M]\right|\right|_1$$
of Equation \eqref{equation:inequalities}. 
\end{proof}

For more details about other interactions between the norms defined above we refer to the work by L\"oh and Pagliantini \cite{LP}.

\section{$\mathcal{R}$-simplicial complexes, $\mathcal{R}$-spaces, foliated simplices and homology}\label{section:R-category}
We briefly recall the basics about $\mathcal{R}$-simplicial complexes and $\mathcal{R}$-spaces. For all the details we refer to Sauer's paper \cite{sauer}. Some material can be also found in the seminal work by Gaboriau \cite{Gaboriau2002} and in Schmidt's PhD thesis \cite{schmidt}.

Throughout, $(X,\mu)$ will be a standard probability spaces and $\mathcal{R}$ will be the orbit equivalence relation associated to an essentially free probability measure preserving action of a countable group $\Gamma$ on $(X,\mu)$, that is 
$$\mathcal{R}\coloneqq \mathcal{R}_{\Gamma\curvearrowright(X,\mu)}=\{(\gamma x, x)\,,\, x\in X\,,\, \gamma\in \Gamma\}\,.$$

\subsection{$X$-spaces and $\mathcal{R}$-simplicial complexes}\label{sec:X:spaces}
A \emph{$X$-space} is a standard Borel space $\Sigma$ equipped with a countable-to-one Borel projection $\pi:\Sigma \to X$. 
Since the projection is sometimes not relevant, we omit it in the notation. We denote by $\Sigma_x$ the fiber of $\pi$ over $x$.

A \emph{$X$-map} between $X$-spaces $\Sigma$ and $\Omega$ is a Borel map $\phi: \Sigma\to \Omega$ fitting in the commutative diagram 
\begin{center}
  \begin{tikzcd}
    \Sigma\arrow{rr}{\phi}\arrow{rd}&&\Omega\arrow{dl}\\
    &X\,.&
  \end{tikzcd}
\end{center}

Given $X$-spaces $\pi:\Sigma\to X$ and $\pi':\Omega\to X$, their \emph{$X$-product} is the fiber product with respect to the projections, namely the $X$-space
$$\Sigma*\Omega=\Sigma{}_{\pi}\times_{\pi'}\Omega=\{(y,z)\in \Sigma\times \Omega\,,\, \pi(y)=\pi'(z)\}\,.$$

Given a $X$-space $\Sigma$, the \emph{natural measure} is the measure $\nu$ defined on a Borel subset $U\subset \Sigma$ as 
\begin{equation}\label{equation:natural:measure}
  \nu(U)=\int_X |U\cap \pi^{-1}(x)| d\mu(x)\,.
\end{equation}

\begin{es}
The Borel equivalence relation $\mathcal{R}$ is an $X$-space. More generally, the set of morphisms $\mathcal{G}^{(1)}$ of any discrete measured groupoid $\mathcal{G}$ with unit space $\mathcal{G}^{(0)}$ is a $\mathcal{G}^{(0)}$-space, and the natural measure $\nu$ coincide with the symmetric quasi-invariant measure on $\mathcal{G}^{(1)}$ \cite{sarti:savini:groupoids,HS}.
\end{es}

An \emph{$\mathcal{R}$-action} on a $X$-space $\Sigma$ is a Borel map 
$$\mathcal{R}* \Sigma\to \Sigma\,,\;\;\; ((y,x),u)\mapsto (y,x)u$$ satisfying the following properties
\begin{itemize}
  \item[(i)] $(y,x) \Sigma_x\subset \Sigma_y$ for all $(y,x)\in \mathcal{R}$.
  \item[(ii)] $(x,x) u=u$ for all $x\in X$ and $u\in \Sigma_x$. 
  \item[(iii)] $(z,y)(y,x) u=(z,x)u$ for all $(z,y),(y,x)\in \mathcal{R}$ and $u\in \Sigma_x$.
\end{itemize}

A \emph{Borel fundamental domain $\mathcal{D}$} for a $\mathcal{R}$-action on $\Sigma$ is a Borel subset of $\Sigma$ that intersects each $\mathcal{R}$-orbit exactly once. 
Given a $\mathcal{R}$-action on $\Sigma$ with Borel fundamental domain $\mathcal{D}$, the measure $\nu|_{\mathcal{D}}$ on $\mathcal{D}$ induces a measure on the space of classes $\mathcal{R}\backslash \Sigma$ equipped with the quotient Borel structure, called \emph{transverse measure}.
 Throughout, we assume that \emph{all} $\mathcal{R}$-actions admit a Borel fundamental domain.

An \emph{$\mathcal{R}$-simplicial complex} is the data of: 
\begin{itemize}
  \item a $X$-space $\Sigma^{(0)}$ equipped with an $\mathcal{R}$-action. 
  \item for each $n\in \mathbb{N}$ a Borel subset $\Sigma^{(n)}\subset \Sigma^{(0)}*\cdots *\Sigma^{(0)}$ of the $(n+1)$-fiber product of $\Sigma^{(0)}$ such that
  \begin{itemize}
    \item $\Sigma^{(n)}$ is invariant under permutations of the coordinates.
\item $(v_0,\cdots ,v_n)\in \Sigma^{(n)}$ implies $v_0 = v_n$.
\item $(v_0,\cdots ,v_n)\in \Sigma^{(n)}$ implies $(v_1,\cdots ,v_n)\in \Sigma^{(n-1)}$.
\item $\mathcal{R}\Sigma^{(n)}=\Sigma^{(n)}$, where the $\mathcal{R}$-action is the obvious one induced by $\mathcal{R}\curvearrowright \Sigma^{(0)}$.
  \end{itemize}
\end{itemize}

Fix a $\mathcal{R}$-simplicial complex $\Sigma$ and a Borel bijection between 
$\mathcal{R}$ and $X\times I$, where $I$ is a countable set. 
%
Moreover, fix a fundamental domain $\mathcal{D}$ for the $\mathcal{R}$-action on $\Sigma^{(0)}$. 
Thanks to the Borel selector theorem \cite{sauer05} we can find a countable partition $\mathcal{D}=\bigsqcup\limits_{n\in \mathbb{N}} \mathcal{D}_n$ of Borel subsets, each of which injects to $X$ via $p_{\Sigma^{(0)}}$. 

Denote by $\Delta(I\times \mathbb{N})$ the $I$-\emph{complete simplicial complex} whose set of vertices is $I$ and having one simplex for each finite collection of vertices. 
Consider now the following composition
\begin{equation}\label{equation_standard_embedding}
\Sigma^{(0)}=\mathcal{R}\mathcal{D}\rightarrow \mathcal{R} \times\mathbb{N}
\rightarrow X\times I\times \mathbb{N}\cong X\times \Delta(I\times \mathbb{N})^{(0)}\,,
\end{equation}
where the first arrow is the assignement $(y,x) \mathcal{D}_n\mapsto ((y,x),n)$ and the second one is induced by the bijection 
$\mathcal{R} \leftrightarrow X\times I$.
Then the map of Equation \eqref{equation_standard_embedding} extends to each $\Sigma^{(n)}$ and the resulting embedding $\Sigma\rightarrow X\times \Delta(I\times \mathbb{N})$ is called a \emph{standard embedding} of $\Sigma$. 

\begin{oss}
  Standard embeddings are neither unique nor canonical in any reasonable sense. Throughout, we will often fix one such map and work with it. As we will see, all our constructions will never depend on this choice. 
\end{oss}

\subsection{$\mathcal{R}$-spaces}\label{sec:R:spaces}
If $\Sigma$ is an $\mathcal{R}$-simplicial complex, we denote by $|\Sigma|$ the disjoint union of the geometric realizations its fibers $\Sigma_x$, that is 
$$|\Sigma|\coloneqq \bigsqcup\limits_{x\in X} |\Sigma_x|\,.$$

A \emph{$\mathcal{R}$-space} is the geometric realization of a $\mathcal{R}$-simplicial complex whose fibers are locally finite for $\mu$-almost every $x\in X$.

A standard embedding induces an embedding 
$$|\Sigma|\rightarrow X\times |\Delta(I\times \mathbb{N})|\,,$$ 
so that $\Sigma$ has a natural standard Borel structure. Here $|\Delta(I\times \mathbb{N})|$ is the geometric realization of the $I$-complete simplicial complex.
Moreover, the $\mathcal{R}$-action on $\Sigma$ descends to a Borel action on $|\Sigma|$.

Let $\mathcal{D}$ be a fundamental domain for the $\mathcal{R}$-action on $\Sigma^{(0)}$.
Then we have an induced action on $\Sigma^{(n)}$ for every $n$.
Given a $n$-simplex in $\Sigma$, it corresponds to a subspace of $|\Sigma|$ which we call a \emph{geometric $n$-simplex}.
The union of all geometric $n$-simplices coming from simplices lying in a fundamental domain of the $\mathcal{R}$-action on $\Sigma^{(n)}$
realizes a Borel fundamental domain for the $\mathcal{R}$-action on $|\Sigma^{(n)}|$. 
The union of such Borel fundamental domains for every $n$ define a Borel fundamental domain for the $\mathcal{R}$-action on $|\Sigma|$.

By definition of Borel fundamental domain and since $\mathcal{R}$ acts trivially on each fiber, the restriction $p:|\Sigma_x|\rightarrow \mathcal{R}\backslash |\Sigma|$ is injective for every $x\in X$.
Thus the image $p(|\Sigma_x|)$ in $\mathcal{R}\backslash |\Sigma|$ can be identified with $|\Sigma_x|$ and is called the \emph{leaf} of $\Sigma$ at $x$. We will denote it as $\mathcal{L}_{\Sigma}(x)$.

Let $\Sigma$ and $\Omega$ be $\mathcal{R}$-simplicial complexes and fix standard embeddings $\Sigma\rightarrow X\times \Delta(I\times \mathbb{N})$ and $\Omega\rightarrow X\times \Delta(I\times \mathbb{N})$. A Borel $X$-map $\phi:|\Sigma|\to |\Omega|$ is of \emph{countable variance} if for any set $A\times K$ with  $A\subset X$
Borel and $K \subset |\Delta(I\times\mathbb{N})|$ compact, there is a countable Borel partition $A=\bigsqcup_
{n\in \mathbb{N}} A_n$
such that $A_n \times  K \subset |\Sigma|$ and each restriction $\phi|_{A_n \times K}$ is a product. 

A \emph{geometric map} between $\mathcal{R}$-spaces is a map of countable variance which is continuous on $\mu$-almost every fiber and proper on $\mu$-almost every fiber.

\begin{es}\label{example:R:spaces}
  let $M$ be a triangulated manifold with universal cover $\widetilde{M}$, $\Gamma=\pi_1(M)$ and $\Gamma\curvearrowright(X,\mu)$ an essentially free probability measure preserving action. Then $\widetilde{M}\times X$ is an $\mathcal{R}$-space for $\mathcal{R}=\mathcal{R}_{\Gamma\curvearrowright (X,\mu)}$
  
  Let $\mathcal{U}$ be an $\mathcal{R}$-cover of $\widetilde{M}\times X$ in the sense of \cite[Definition 2.27]{sauer}. Then the \emph{nerve} $\mathcal{N}(\mathcal{U})$ is an $\mathcal{R}$-simplicial complex and the \emph{nerve map} $\widetilde{M}\times X\to |\mathcal{N}(\mathcal{U})|$ is a geometric map \cite[Lemma 2.35]{sauer}. 
\end{es}

\subsection{Foliated singular simplices and homology}\label{sec:foliated}
In this section we introduce simplices on $\mathcal{R}$-spaces. We warn the reader that the definition of singular foliated simplices and all the constructions up to Section \ref{sec:funct} could be done in the more general setting of $\mathcal{R}$-spaces in the sense of Schmidt \cite{schmidt}. An $\mathcal{R}$-space for Schmidt is a $X$-space that carries a compatible $\sigma$-compact, second countable Hausdorff topology. However, functoriality of our homology theory holds only for geometric maps, and the latter are defined through the underlying simplicial structure. For this reason, we prefer to follow faithfully Sauer and to work with $\mathcal{R}$-spaces as defined in the previus Section. 

Let $\Sigma$ be an $\mathcal{R}$-space. Fix $k\in \mathbb{N}$. We consider the bundle of singular $k$-simplices
$$\Sing_k(\Sigma)\coloneqq \bigsqcup_{x\in X} \Sing_k(\Sigma_x)$$
where $\Sing_k(\Sigma_x)$ is the set of singular $k$-simplices of $\Sigma_x$. The bundle comes with a natural projection onto $X$ and with a $\mathcal{R}$-action. Moreover, an $\mathcal{R}$-fundamental domain $\mathcal{D}\subset \Sigma$ induces a fundamental domain of such action. 
The choice of a standard embedding $\Sigma\rightarrow X\times \Delta(I\times \mathbb{N})$ descend to a map $\Sing_n(\Sigma)\rightarrow X\times \Sing_n(\Delta(I\times \mathbb{N}))$. 

\begin{defn}
  A subset $W\subset \Sing_k(\Sigma)$ is \emph{admissible} if its image under the standard embedding is contained in a product $C\times U$ with $C\subset \Sing_k(\Delta(I\times \mathbb{N}))$ is countable and $U\subset X$ is Borel. 

  A subset $W\subset \mathcal{R}\backslash\Sing_k(\Sigma)$ is \emph{admissible} if its preimage in $\Sing_k(\Sigma)$ is admissible. 
\end{defn}

The above definition does not depend on the choice of the standard embedding \cite[Lemma 3.5]{sauer}.

We consider the $\sigma$-algebra on $\Sing_k(\Sigma)$ and on $\mathcal{R}\backslash \Sing_k(\Sigma)$ generated by admissible subsets \cite[Remark 3.3]{sauer}.
Admissible subsets of $ \Sing_k(\Sigma)$ are $X$-spaces, whence they comes with a natural measure $\nu$ (Equation \eqref{equation:natural:measure}). The latter descends to a measure on admissible subsets of $\mathcal{R}\backslash \Sing_k(\Sigma)$, called transverse measure.

\begin{es}\label{example:decomposition}
Let $\Sigma=\widetilde{M}\times X$ as in Example \ref{example:R:spaces} and retain the notation fixed in Setup \ref{setup}. Given an admissible $U\subset \Sing_k(\widetilde{M})\times X$ 
we can write 
$$\nu(U)
=\int_X |U\cap (\Sing_k(\widetilde{M})\times \{x\})|d\mu(x)\,.$$
Moreover, the natural measure $\overline{\nu}$ on an admissible subset $W\subset  \Gamma\backslash(\Sing_k(\widetilde{M})\times X)$ can be defined via the identification of $\Gamma\backslash(\Sing_k(\widetilde{M})\times X)$ with the Borel fundamental domain $\mathcal{D}_k=D_k\times X$ of $\Gamma\curvearrowright \Sing_k(\widetilde{M})\times X$.
Therefore, we have the following explicit formula
$$\overline{\nu}(W)\coloneqq \int_X\sum\limits_{s\in D_k} |\pi^{-1}(W)\cap \{(s,x)\}|d\mu(x)\,.$$
It is easy to check that $\overline{\nu}$ does not depend on the choice of the fundamental domain $\mathcal{D}_k$.

Let $\mathcal{D}_k$ be any fundamental domain of $\Gamma\curvearrowright \Sing_{k}(\widetilde{M})\times X$. Then $\nu$ decomposes as  
  \begin{equation}\label{equation:decomposition:natural:measure}
  \nu=\sum\limits_{\gamma\in\Gamma}\gamma_*\nu|_{\mathcal{D}_k}\,,
  \end{equation}
  where $\gamma_*\nu(U)=\nu(\gamma^{-1} U)$ for every admissible 
$U\subset \Sing_{k}(\widetilde{M})\times X$. 
\end{es}

\begin{defn}[singular foliated simplex, \cite{sauer}]\label{definition:singular:simplex}
  A countable-to-one map $\sigma:  A\to \mathcal{R}\backslash \Sing_k(\Sigma)$ where $A\subset \mathbb{R}$ is a subset of finite Lebesgue measure $m_{\mathcal{L}}(A)<+\infty$ is a \emph{singular foliated $k$-simplex} if
  \begin{itemize}
    \item[(i)] $\Ima(\sigma)$ is admissible and $\sigma:A\to \Ima(\sigma)$ is a $\Ima(\sigma)$-space.
    \item[(ii)] $$m_{\mathcal{L}}(A)=\int_{\Ima(\sigma)} |\sigma^{-1}(s)| d\overline{\nu}(s)\,.$$
    \item[(iii)] for $\mu$-almost every $x\in X$ the set $\Ima(\sigma)\cap \Sing_k(\mathcal{L}_\Sigma(x))$ is locally finite in $\mathcal{L}_\Sigma(x)$. 
  \end{itemize}

  The \emph{norm} (or \emph{weight}) of $\sigma$ is $|\sigma|\coloneqq m_{\mathcal{L}}(A)$. 
\end{defn}

We consider the set of singular foliated $k$-simplices $S_k(\mathcal{R}\backslash\Sigma)$ and the free $\mathbb{R}$-module generated by $S_k(\mathcal{R}\backslash\Sigma)$ 
$$\mathcal{C}_k(\mathcal{R}\backslash \Sigma;\mathbb{R})\coloneqq \mathbb{R}[S_k(\mathcal{R}\backslash\Sigma)]\,.$$
Since we want to consider objects up to null sets, we need to identify simplices (and chains) that ``coincide almost everywhere". To this end, we follow again Sauer. Given $\sigma\in S_k(\mathcal{R}\backslash\Sigma)$ we define its \emph{multiplicity function} as
$$\omega_{\sigma}:\mathcal{R}\backslash\Sing_k(\Sigma) \to \mathbb{Z}\,,\;\;\; s\mapsto |\sigma^{-1}(s)|\,.$$ Then, by linearity, we can define $\omega_\rho$ for every $\rho\in\mathcal{C}_k(\mathcal{R}\backslash \Sigma;\mathbb{R})$. Given a chain $\rho\in \mathcal{C}_k(\mathcal{R}\backslash \Sigma;\mathbb{R})$, by condition (iii) of Definition  \ref{definition:singular:simplex}, the formal sum $$\rho(x)\coloneqq \sum\limits_{s\in \Ima(\rho)} \omega_{\rho}(s) s$$ defines a locally finite singular chain in $\Ck^{\text{lf}}(\mathcal{L}_\Sigma(x);\mathbb{R})$ for almost every $x\in X$. The face operators of the $\Sigma_x$'s induce face operators  $d_k^i:\Sing_k(\Sigma)\to \Sing_{k-1}(\Sigma)$. Moreover, because the latter are compatible with the $\mathcal{R}$-action, we obtain boundary operators $\partial_k:\mathcal{C}_k(\mathcal{R}\backslash \Sigma;\mathbb{R})\to \mathcal{C}_{k-1}(\mathcal{R}\backslash \Sigma;\mathbb{R})$ defined as $\partial_k\coloneqq \sum\limits_{i=0}^k (-1)^i \partial_k^i$, where $\partial_k^i\sigma$ is given by post-composing $\sigma$ with $d_k^i$. In particular if $\sigma\in S_k(\mathcal{R\backslash}\Sigma)$ then 
\begin{equation}\label{eq:boundary}
  \partial_{k}\sigma=\sum\limits_{i=0}^{k}  (-1)^{i} \partial^i_{k}\sigma\,.
\end{equation}
Moreover, if $d^x_k:\Ck^{\text{lf}}(\mathcal{L}_\Sigma(x);\mathbb{R})\to\Cm_{k-1}^{\text{lf}}(\mathcal{L}_\Sigma(x);\mathbb{R}) $ is the boundary operator on the leaf at $x$, the following diagram commutes
\begin{center}
  \begin{tikzcd}
    \mathcal{C}_k(\mathcal{R}\backslash \Sigma;\mathbb{R}) \arrow{d}[swap]{\partial_k}\arrow{r}{\rho\mapsto \rho(x)}& \Ck^{\text{lf}}(\mathcal{L}_\Sigma(x);\mathbb{R})\arrow{d}{d_k^x}\\
    \mathcal{C}_{k-1}(\mathcal{R}\backslash \Sigma;\mathbb{R})\arrow{r}{\rho\mapsto \rho(x)}& \textup{C}_{k-1}^{\text{lf}}(\mathcal{L}_\Sigma(x);\mathbb{R})
  \end{tikzcd}
\end{center}
for every $x\in X$ and in any degree. Hence the subsets
$$\mathcal{C}^{\text{a.e.}}_k(\mathcal{R}\backslash \Sigma;\mathbb{R})\coloneqq \left\{\rho \in \mathcal{C}_k(\mathcal{R}\backslash \Sigma;\mathbb{R})\,,\, \rho(x)=0 \; \text{for a.e. } x\in X\right\}$$
form a well-defined subcomplex of $(\mathcal{C}_k(\mathcal{R}\backslash \Sigma;\mathbb{R}),\partial_k)$.
\begin{defn}\label{definition:homology}
  The \emph{real singular foliated homology} of $\Sigma$ is the homology of the complex $\Ck(\mathcal{R}\backslash \Sigma;\mathbb{R})\coloneqq \mathcal{C}_k(\mathcal{R}\backslash \Sigma;\mathbb{R})/\mathcal{C}^{\text{a.e.}}_k(\mathcal{R}\backslash \Sigma;\mathbb{R})$, that is 
$$\Hm_k(\mathcal{R}\backslash \Sigma;\mathbb{R})\coloneqq \Hm_k(\text{C}_{\bullet}(\mathcal{R}\backslash \Sigma;\mathbb{R}),\partial_\bullet)\,.$$
\end{defn}

The complex $(\text{C}_{\bullet}(\mathcal{R}\backslash \Sigma;\mathbb{R}),\partial_\bullet)$ becomes a normed chain complex when endowed with the norm obtained by extending the one defined on a simplex $\sigma\in S_k(\mathcal{R}\backslash\Sigma)$ as
\begin{equation}\label{equation:norm}
  |\sigma|\coloneqq \int\limits_{\Ima(\sigma)} |\omega_{\sigma}|=\overline{\nu}(|\omega_{\sigma}|)\,.
\end{equation}
Indeed, $\mathcal{C}^{\text{a.e.}}_k(\mathcal{R}\backslash \Sigma;\mathbb{R})$ is precisely the subspace of zero-norm chains. 
Such norm descends in homology to a seminorm, still denoted by $|\cdot |$.

\begin{oss}\label{remark:diff}
A notion of singular simplices for $\mathcal{R}$-spaces was introduced also by Schmidt in his PhD thesis \cite{schmidt}.
A singular $k$-simplex on a $\mathcal{R}$-space $\Sigma$ in the sense of Schmidt is a $X$-map 
\begin{equation}\label{eq:schmidt}
  \widetilde{\sigma}:X_{\widetilde{\sigma}}\to \Sing_k(\Sigma) 
\end{equation}
where $X_{\widetilde{\sigma}}\subset X$ is Borel.
 One may think that the composition of $\widetilde{\sigma}$ with the projection $p:\Sing_k(\Sigma) \to \mathcal{R}\backslash\Sing_k(\Sigma)$ could give a foliated singular simplex as in Definition \ref{definition:singular:simplex}. However, a priori $\sigma=p\circ\widetilde{\sigma}$ could violate both condition (i) and (iii): indeed, $\Ima(\sigma)$ may not be admissible and it may induce non-locally finite chains on almost every leaf. Explicit examples of both phenomena can be easily constructed. For instance, take $\Sigma=\widetilde{M}\times X$ the $X$-space equipped with the projection $\pi_2$ on the second component. If $\pi_1$ is the projection on the first factor and $\widetilde{\sigma}:X_{\widetilde{\sigma}}\mapsto\Sing_k( \widetilde{M})\times X$ is such that the set $\{\pi_1\circ \widetilde{\sigma}(t)\,,\, t\in X_{\widetilde{\sigma}}\}\subset \Sing_k(\widetilde{M})$ is uncountable, then $\sigma=p\circ\widetilde{\sigma}$ does not satisfies condition (i).
In general, we do not know whether any simplex as in Definition \ref{definition:singular:simplex} can be obtained as the projection of a map as in \eqref{eq:schmidt}. In any case, no one-to-one correspondence holds.

Another difference with \cite{schmidt} lies in the choice of coefficients. In fact, Schmidt's homology incorporates the dynamical data already at the level of the coefficients: his chains of simplices have coefficients in $\Linf(X,\Lambda)$, where $\Lambda$ is a subring of $\mathbb{C}$ (tipically $\mathbb{Z}, \mathbb{R}$ or $\mathbb{C}$).
In contrast, our approach only considers real chains; the dynamical data will appear exclusively through the simplices themselves, rather than in the coefficients.
When $\Lambda=\mathbb{C}$ Schmidt proves that the homology of $\widetilde{M}\times X$ is isomorphic to the singular homology of $M$ with coefficients in $\Linf(X,\mathbb{C})$. We do not expect a similar result in our framework.

In conclusion, although we do not a priori exclude a substantial overlap between Schmidt's theory and ours, the two approaches differ in several fundamental aspects, reflecting the distinct purposes for which they were developed.
\end{oss}

\subsection{Functoriality}\label{sec:funct}
We prove that foliated singular homology is a functor from the category of $\mathcal{R}$-spaces and geometric maps to the category of normed vector spaces. Since the results below are straight consequences of Sauer's work, we include only the proof of novelties and we refer to \cite{sauer} for all the details.

\begin{prop}[Sauer]
  Let $\Sigma,\Omega$ be $\mathcal{R}$-spaces and $\phi:\Sigma\to \Omega$ be a geometric map. Then $\phi$ descends to maps $\Cm_k(\phi):\Cm_k(\mathcal{R}\backslash \Sigma;\mathbb{R})\to \Cm_k(\mathcal{R}\backslash \Omega;\mathbb{R})$ and $\Hm_k(\phi):\Hm_k(\mathcal{R}\backslash \Sigma;\mathbb{R})\to \Hm_k(\mathcal{R}\backslash \Omega;\mathbb{R})$.
\end{prop}
\begin{proof}
  The proof is identical to the one of \cite[Theorem 3.12]{sauer}. 

\end{proof}

\begin{lemma}
  Let $\Sigma,\Omega$ be $\mathcal{R}$-spaces and $\phi:\Sigma\to \Omega$ be a geometric map. Then $\Cm_k(\phi)$ and $\Hm_k(\phi)$ do not increase the norm.
\end{lemma}
\begin{proof}
Consider a foliated $k$-simplex $\sigma: A\to \mathcal{R}\backslash \Sing_k(\Sigma)$. By definition, $\left|\Cm_k(\phi)(\sigma)\right|=m_{\mathcal{L}}(A)=|\sigma|$. By linearity $\left|\sum\limits_{i=1}^k \Cm_k(\phi)(\sigma_i)\right|=\sum\limits_{i=1}^k|a_i|\left|\Cm_k(\phi)(\sigma_i)\right|=\sum\limits_{i=1}^k|a_i||\sigma_i|$, thus  $\Cm_k(\phi)$ is isometric. 
The second part follows since the induced map in homology of any norm non-increasing map has norm at most one.  
\end{proof}

\section{Foliated homology and simplicial volume}\label{section:foliated}
The goal of this section is to focus on the construction of Section \ref{section:R-category} in the specific framework of Setup \ref{setup}, namely when $\Sigma=\widetilde{M}\times X$.
We will give the definition of \emph{foliated fundamental class} and study its norm, comparing it with the simplicial volume (Theorem \ref{theorem:simplicial:volume}).

Retain the setting of Setup \ref{setup}.
First of all, we adapt a construction by Sauer \cite{sauer} to define a chain map 
$ \Ck(M;\Linf(X,\mathbb{Z})\otimes_{\mathbb{Z}\Gamma} \mathbb{R})\to \Ck(\Gamma \backslash(\widetilde{M}\times X); \mathbb{R})\,.$

\begin{lemma}[Sauer]\label{lemma:lambda}
  The map 
  $$\lambda_k: \Ck(M;\Linf(X,\mathbb{Z})\otimes_{\mathbb{Z}\Gamma} \mathbb{R})\to \Ck(\Gamma \backslash(\widetilde{M}\times X); \mathbb{R})\,$$
uniquely determined by 
\begin{gather*}
 s\otimes [\chi_A]\mapsto \left( j^{-1}(A)\to \Gamma\backslash(\Sing_k(\widetilde{M})\times X)\,,\;t\mapsto [s,j(t)] \right)
\end{gather*} 
where $j:[0,1]\to X$ is a measure preserving isomorphism, $A\subset X$ is a Borel subset and $s_i\in \Sing_k(\widetilde{M})$, does not depend on $j$ and is an isometric chain map. 
\end{lemma}
\begin{proof}
  The proof is identical to \cite[Lemma 3.23]{sauer}. The only difference concerns the isometricity statement. To see this, is is sufficient to notice that the image of $s\otimes [\chi_A]$ has norm $\mu(j^{-1}(A))=m_{\mathcal{L}}(A)=|s\otimes [\chi_A]|$. Since $\lambda_k$ is the linear extension of maps as above and any $c\in \Ck(M;\Linf(X,\mathbb{Z})\otimes_{\mathbb{Z}\Gamma} \mathbb{R})$ has a representative of the form $c=\sum\limits_{i=1}^k s_i\otimes \chi_{A_i}\otimes a_i$ with $a_i\in \mathbb{R}$, $A_i\subset X$ Borel, $s_i\in \Sing_k(\widetilde{M})$ with $p(s_i)\neq p(s_j)$ when $i\neq j$, we have
  $$\left|\lambda_k\left(\sum\limits_{i=1}^k s_i\otimes [\chi_{A_i}]\otimes a_i\right)\right|=\sum\limits_{i=1}^k | a_i|\mu(A_i)=\left\bracevert\sum\limits_{i=1}^k s_i\otimes [\chi_{A_i}]\otimes a_i\right\bracevert^X_{\mathbb{R}}\,.$$
  This shows that $\lambda_k$ is isometric.
\end{proof}

We denote by $\Lambda_k: \Hk(M;\Linf(X,\mathbb{Z})\otimes_{\mathbb{Z}\Gamma} \mathbb{R})\to \Hk(\Gamma\backslash(\widetilde{M}\times X);\mathbb{R})$ the induced map in homology. 

\begin{defn}\label{def:fund}
  The \emph{$X$-foliated fundamental class} of $M$ is
  $$[\Gamma\backslash(\widetilde{M}\times X)]\coloneqq \Lambda_n([M]^X_\mathbb{R})\in \Hn(\Gamma \backslash(\widetilde{M}\times X);\mathbb{R})\,.$$
  A \emph{$X$-foliated fundamental cycle} of $M$ is a representative of $[\Gamma\backslash(\widetilde{M}\times X)]$ in $\Cn(\Gamma \backslash(\widetilde{M}\times X);\mathbb{R})$.
\end{defn}

\begin{oss}
  Any foliated fundamental cycle can be restricted on the leaves $\mathcal{L}(x)$ and interpreted as a ``measurable family of locally finite fundamental cycles". Indeed, it turns out that, for almost every $x\in X$, we obtain a locally finite foliated fundamental cycle in $\Cm_n^{\text{lf}}(\widetilde{M};\mathbb{R})$. To see this, for every $x\in X$ we consider the maps
  $$ \phi_x:\Cn(\Gamma \backslash(\widetilde{M}\times X);\mathbb{R})\to \Cn^{\text{lf}}(\mathcal{L}_\Sigma(x);\mathbb{R})= \Cn^{\text{lf}}(\widetilde{M};\mathbb{R})\,,\;\;\; \rho \mapsto \rho(x)$$
  and 
  $$\varphi_x: \Cn(\widetilde{M}) \otimes_{\mathbb{Z}\Gamma} \Linf(X,\mathbb{Z})\otimes_{\mathbb{Z}\Gamma} \mathbb{R}\to\Cn^{\text{lf}}(\widetilde{M};\mathbb{R})\,,\;\;\; \sum\limits_{i=1}^m s_i\otimes [f_i] \mapsto \sum\limits_{i=1}^m \sum\limits_{\gamma\in \Gamma} f_i(\gamma^{-1} x) \gamma \cdot s_i\,,$$
  where the latter is the linear extension of the map defined by Frigerio, L\"{o}h, Pagliantini and Sauer \cite[Lemma 2.5]{FLPS}. 
  Then for almost every $x\in X$ we have a commutative diagram 
\begin{center}
  \begin{tikzcd}
    \Cn(\widetilde{M}) \otimes_{\mathbb{Z}\Gamma} \Linf(X,\mathbb{Z})\otimes_{\mathbb{Z}\Gamma} \mathbb{R}\arrow{rr}{\lambda_n}  \arrow[swap]{rd}{\varphi_x} & & \Cn(\Gamma \backslash(\widetilde{M}\times X);\mathbb{R})\arrow{ld}{\phi_x}\\
   & \Cn^{\text{lf}}(\widetilde{M};\mathbb{R})\,.&
  \end{tikzcd}
\end{center}
If $\rho\in \Cn(\Gamma \backslash(\widetilde{M}\times X);\mathbb{R})$ is a $X$-foliated fundamental cycle then $\rho(x)$ is a well-defined locally finite fundamental cycle of $\widetilde{M}$ for almost every $x\in X$.
This is a direct consequence of \cite[Lemma 2.5]{FLPS} or of the following considerations. First of all, we denote by $B(X,\mathbb{Z})$ the set of $\mathbb{Z}$-valued bounded measurable functions on $X$ and by $N(X,\mathbb{Z})$ the subset of functions vanishing almost everywhere, so that $\Linf(X,\mathbb{Z})=B(X,\mathbb{Z})/N(X,\mathbb{Z})$. In this setting we can find $c\in \Cn(\widetilde{M})\otimes B(X,\mathbb{Z})$, $b\in \text{C}_{n+1}(\widetilde{M})\otimes B(X,\mathbb{Z})$ and $z\in \text{C}_{n}(\widetilde{M})\otimes N(X,\mathbb{Z})$ such that $\rho=\lambda_n(c+\partial \ell(b) + \ell(z))$, $c=\ell(i(w))$ and $w$ is an integral fundamental cycle in $\Cn(\widetilde{M})\otimes_{\mathbb{Z}\Gamma} \mathbb{Z}$. Here $\lambda_n$ is the same of Lemma \ref{lemma:lambda} but defined on $B(X,\mathbb{Z})$ and $\ell$ and $i$ are the same as in Diagram \ref{diagram}. 
We therefore have
\begin{align*}
  \rho(x)&=\lambda_n(c+\partial \ell(b) + \ell(z))(x)\\
  &=\lambda_n(\ell(i(w)))(x)+\lambda_n(\partial \ell(b))(x) + \lambda_n(\ell(z))(x)\\
  &=w+\partial(\lambda_n(\ell(b))(x))\,,
\end{align*}
where we exploited the linearity of $\lambda_n$ and of $\varphi_x$, the fact that $\lambda_n$ commutes with $\partial$ and the fact $\lambda_n(z)\in \mathcal{C}_n^{\text{a.e.}}(\widetilde{M}\times X)$. 
\end{oss}

Since $\lambda_n$ is isometric, $\Lambda_n$ does not increase the norm. Thus, combining with Lemma \ref{lemma:simplicial:volume} we have
\begin{equation}\label{equation:inequality:foliated}
  |[\Gamma\backslash(\widetilde{M}\times X)]|\leq \left\bracevert[M]_{\mathbb{R}}^X\right\bracevert^X_{\mathbb{R}} = ||M||\,.
\end{equation}
The opposite inequality, which is the non-trivial part of Theorem \ref{theorem:simplicial:volume}, follows by a technical construction of Sauer. We prepare the proof with the following
\begin{lemma}[Sauer]\label{lemma:sauer}
  Fix $\varepsilon>0$. Let $\rho \in \Cn(\Gamma \backslash(\widetilde{M}\times X);\mathbb{R})$. Then there exists $c\in \Cn(M;\Linf(X,\mathbb{Z})\otimes_{\mathbb{Z}\Gamma} \mathbb{R})$ such that $|\rho -\lambda_n(c)|< \varepsilon$.
\end{lemma}
\begin{proof}
  The proof is inspired from \cite[Lemma 3.17]{sauer}.
  Without loss of generalities we may assume that $\rho$ is a single simplex $A\to \Gamma\backslash \Sing_n(\widetilde{M})\times X$. Fix a fundamental domain $D_n$ of $\Gamma\curvearrowright \Sing_n(\widetilde{M})$. Since $\Ima(\rho)$ is admissible, after identifying $\Gamma\backslash \Sing_n(\widetilde{M})\times X$ with $D_n\times X$ we may assume that $\Ima(\rho)$ is contained in $C\times X$ where $C=\{s_1,s_2,\ldots \}$ is a countable set of singular $n$-simplices in $\widetilde{M}$. 
  Then we consider the direct system of sets
  $$W_{j,k}\coloneqq \{(s_i,x)\,|\, x\in X\,,\, |\rho^{-1}(s_i,x)|\leq j \,,\, i\leq k\}$$
  and their preimage 
  $A_{j,k}=\rho^{-1}(W_{j,k})$. 
  Since the $W_{j,k}$'s form a direct system with $\bigcup_{j,k} W_{j,k}=\Ima(\rho)$, for sufficiently large $j,k$ one has $m_{\mathcal{L}}(A\setminus A_{j,k})<\varepsilon$.
  Then we have 
  $$|\rho-\rho|_{A_{j,k}}|=|\rho|_{A\setminus A_{j,k}}|=m_{\mathcal{L}}(A\setminus A_{j,k})<\varepsilon\,.$$
Indeed 
\begin{align*}
\left|\rho-\rho_{|A_{j,k}}\right|&=\int_{\Ima(\rho-\rho|_{A_{j,k}}) }
|\omega_{\rho-\rho_{|A_{j,k}}}([s,x])| d\overline{\nu}([s,x])\\
&=\int_{(C\times X)\setminus W_{j,k}}
|\rho^{-1}([s,x])|d\overline{\nu}([s,x])=\left|\rho|_{A\setminus A_{j,k}}\right|\,.
\end{align*}
We now show that $\rho|_{A_{j,k}}$ is in the image of $\lambda_n$. To this end, we exploit the Selection Theorem \cite[Lemma 3.1]{sauer05} to find a Borel partition of $$A_{j,k}=\bigsqcup_{1\leq r,q\leq j,k} B_{r,q}$$ such that $\Ima(\rho|_{B_{r,q}})\subset \{s_q\}\times X$ and $\rho|_{B_{r,q}}:B_{r,q}\to \{s_q\}\times X$ is injective. 
Setting 
$$c=\sum_{1\leq r,q\leq j,k} [\chi_{B_{r,q}}]\otimes s_q\in \Cn(\widetilde{M})\otimes_{\mathbb{Z}\Gamma} \Linf(X, \mathbb{Z})$$
we immediately get 
$\lambda_n(c)=\rho|_{B_{r,q}}$. 
The general case, namely when $\rho$ is a chain, follows by extending $\mathbb{R}$-linearly the above argument.
\end{proof}

We are now ready to show that Equation \eqref{equation:inequality:foliated} is in fact an equality.
\begin{rec_thm}[\ref{theorem:simplicial:volume}]
  Let $M$ be a triangulated oriented closed connected $n$-manifold. Let $p: \widetilde{M}\to M$ be its universal cover and consider an essentially free measure preserving action $\Gamma=\pi_1(M)\curvearrowright (X,\mu)$ on a standard Borel probability space. Then 
  $$||M||= |[\Gamma\backslash(\widetilde{M}\times X)]|\,.$$
\end{rec_thm}
\begin{proof}
  The structure of the proof follows \cite[Lemma 3.24]{sauer}, namely we prove that $\Lambda_n$ is isometric.
  Since it is norm non-increasing, we only needs to show that for every cycle $c\in \Cn(M;\Linf(X,\mathbb{Z})\otimes_{\mathbb{Z}\Gamma} \mathbb{R})$ and every $\rho\in \Cn(\Gamma \backslash(\widetilde{M}\times X);\mathbb{R})$ such that $\lambda_n(c)$ and $\rho$ are cohomologous, 
  one can find a chain $b\in \textup{C}_{n+1}(M;\Linf(X,\mathbb{Z})\otimes_{\mathbb{Z}\Gamma} \mathbb{R})$ with $\left\bracevert c+\partial_{n+1} b\right\bracevert^X_{\mathbb{R}}< |\rho|+\varepsilon$ for every $\varepsilon>0$. 
To this end, we denote $\beta \in \textup{C}_{n+1}(\Gamma\backslash(\widetilde{M}\times X);\mathbb{R})$ the element such that $\lambda_n(c)+\partial_{n+1}\beta=\rho$. Then by Lemma \ref{lemma:sauer} there exists $b\in \textup{C}_{n+1}(M;\Linf(X,\mathbb{Z})\otimes_{\mathbb{Z}\Gamma} \mathbb{R})$ such that 
$$\left\bracevert\beta-\lambda_{n+1}(b)\right\bracevert^X_{\mathbb{R}}< \varepsilon/(n+2)\,.$$
We compute 
\begin{align*}
  \left\bracevert c+\partial_{n+1} b\right\bracevert^X_{\mathbb{R}} &= |\lambda_n(c+\partial_{n+1} b)|\\
  &= |\lambda_n(c)+\lambda_n(\partial_{n+1} b)|\\
  &= |\rho-\partial_{n+1}\beta+\partial_{n+1} (\lambda_{n+1} (b))|\\
  &\leq |\rho|+(n+2)|\beta-\lambda_{n+1} (b)|< |\rho|+\varepsilon\,.
\end{align*}
In the above computation we first used the fact that $\lambda_n$ is isometric, then its linearity and finally we moved from the third line to the last one exploiting the triangular inequality and the obvious bound 
$||\partial_{n+1} ||_{\text{op}}\leq n+2$.
This concludes the proof.
\end{proof}

\section{A dual point of view}\label{section:foliated:cohomology}
In this section we take the topological dual of the complex 
$\Cm_\bullet(\Gamma\backslash(\widetilde{M}\times X);\mathbb{R})$ and we study its cohomology. Inspired by the isometric isomorphism between the bounded cohomology of topological spaces and the one of their fundamental group \cite{Ivanov,vbc}, we connect the bounded cohomology foliated simplices with probability measure preserving actions and (Theorem \ref{theorem:bounded:cohomology}), more generally, with the bounded cohomology of measured groupoids \cite{sarti:savini:groupoids}. As a by-product, we obtain a vanishing criterion for simplicial volume (Corollary \ref{corollary:vanishing}). 

\subsection{Singular foliated bounded cohomology.}
We consider the topological dual of $\Ck(\Gamma \backslash(\widetilde{M}\times X); \mathbb{R})$, that is 
$$\Cb^k(\Gamma\backslash(\widetilde{M}\times X);\mathbb{R})\coloneqq \left\{f: \Ck(\Gamma \backslash(\widetilde{M}\times X); \mathbb{R})\to \mathbb{R}\,,\, \sup\limits_{ |\sigma|=1} |f(\sigma)|<+\infty\right\}\,.$$
Define the coboundary operator $\delta^k:\Cb^k(\Gamma\backslash(\widetilde{M}\times X);\mathbb{R})\to \Cb^{k+1}(\Gamma\backslash(\widetilde{M}\times X);\mathbb{R})$ as the dual of $\partial_k$. 
Each $\Cb^k(\Gamma\backslash(\widetilde{M}\times X);\mathbb{R})$ is endowed with the norm 
$$||f||_{\infty}\coloneqq  \sup\limits_{|\sigma|=1} |f(\sigma)|\,.$$

\begin{rec_defn}[\ref{def:foliated:bc}]
  The \emph{singular foliated bounded cohomology} of $M$ is 
  $$\Hb^k(\Gamma \backslash(\widetilde{M}\times X); \mathbb{R})\coloneqq \Hm^k(
    \Cb^{\bullet}(\Gamma \backslash(\widetilde{M}\times X); \mathbb{R}),\delta^{\bullet})\,.$$
\end{rec_defn}

Our next goal is to prove a relation between singular foliated bounded cohomology and classic singular bounded cohomology. To this end we exploit the homological algebraic technology described in Section \ref{sec:homo}, in particular the Fundamental Lemma of Homological Algebra (Lemma \ref{lemma:fundamental}). Precisely, we need to show that the above bounded cohomology can be computed with a \emph{strong resolution} by \emph{relatively injective} $\Gamma$-modules of some $\Gamma$-module. 
The idea comes from classic singular bounded cohomology, where one has the obvious but fundamental identification 
$$\Cb^k(M;\mathbb{R})\cong \Cb^k(\widetilde{M};\mathbb{R})^{\Gamma}\,$$
and the modules on the right hand side form a strong resolution of $\mathbb{R}$ by relatively injective $\Gamma$-modules \cite[Chapter 5]{miolibro}.
In our case the new resolution is obtained introducing a non-equivariant version of foliated simplices and the associated chain complex.

\begin{defn}\label{definition:singular:simplex:non:eq}
  A countable-to-one map $\sigma:  A\to \Sing_k(\widetilde{M})\times X$ where $A\subset \mathbb{R}$ is a subset of finite Lebesgue measure is a \emph{non-equivariant singular foliated $k$-simplex} if
  \begin{itemize}
    \item[(i)] $\Ima(\sigma)$ is admissible and $\sigma:A\to \Ima(\sigma)$ is measurable;
    \item[(ii)] $$m_{\mathcal{L}}(A)=\int_{\Ima(\sigma)} |\sigma^{-1}(s,x)| d\nu(s,x)\,;$$
    \item[(iii)] for almost every $x\in X$ the set 
    $$\bigcup_{\gamma\in \Gamma}\gamma^{-1}\left( \Ima(\sigma)\cap (\widetilde{M}\times\{\gamma x\} )\right)$$ is locally finite in $\widetilde{M}\times \{x\}$. 
  \end{itemize}

  The \emph{norm} (or the \emph{weight}) of $\sigma$ is $|\sigma|\coloneqq m_{\mathcal{L}}(A)$. 
\end{defn}

We consider the set $S_k(\widetilde{M}\times X)$ of non-equivariant singular foliated $k$-simplices and the free $\mathbb{R}$-module 
$$\mathcal{C}_k(\widetilde{M}\times X;\mathbb{R})\coloneqq \mathbb{R}[S_k(\widetilde{M}\times X)]\,.$$

As a consequence of condition (iii), for almost every $x$ the set 
$\Ima(\sigma)\cap \widetilde{M}\times \{x\}$ is locally finite. Thus one can exploit the boundary operator on each $\widetilde{M}\times \{x\}$ to define a boundary operator $\partial_k$ on $\mathcal{C}_k(\widetilde{M}\times X;\mathbb{R})$. 

The chain complex $(\mathcal{C}_k(\widetilde{M}\times X;\mathbb{R}),\partial_k) $ comes with a natural $\Gamma$-action defined as 
$$(\gamma\cdot \rho) (t)=\gamma \cdot \rho(t)\,,$$
where on the right $\Gamma$ is acting diagonally.
The same technique used in the equivariant setting allows to quotient out by the subcomplex of chains vanishing almost everywhere. Precisely we define the \emph{multiplicity function} of a simplex $\sigma$ as 
$$\omega_{\sigma}: \Sing_k(\widetilde{M})\times X\to \mathbb{Z}\,,\;\;\; \omega_{\sigma}(s,x)=|\sigma^{-1}(s,x)|\,,$$
we extend by linearity to any $\rho\in\mathcal{C}_k(\widetilde{M}\times X; \mathbb{R})$ and consider the subspace $$\mathcal{C}_k^{\text{a.e.}}(\widetilde{M}\times X; \mathbb{R})=\left\{\rho\in \mathcal{C}_k(\widetilde{M}\times X; \mathbb{R})\,,\, \sum\limits_{(s,x)\in \Ima(\sigma)}\omega_{\rho}(s,x)s=0 \; \text{ for a.e. } x\in X\right\} \,.$$ 
The same argument of Section \ref{section:foliated} shows that 
$(\mathcal{C}_\bullet^{\text{a.e.}}(\widetilde{M}\times X; \mathbb{R}),\partial_\bullet)$ is a subcomplex of $(\mathcal{C}_\bullet^(\widetilde{M}\times X; \mathbb{R}),\partial_\bullet)$.
Then we define $$\Ck(\widetilde{M}\times X; \mathbb{R})\coloneqq \mathcal{C}_k(\widetilde{M}\times X;\mathbb{R})/\mathcal{C}_k^{\text{a.e.}}(\widetilde{M}\times X;\mathbb{R})\,$$ so that $(\Cm_\bullet(\widetilde{M}\times X; \mathbb{R}),\partial_\bullet)$ is a complex of normed $\Gamma$-spaces with the linear extension of the norm defined on a simplex $\sigma\in S_k(\widetilde{M}\times X)$ as
$$|\sigma|\coloneqq \int\limits_{\Ima(\sigma)}|\omega_{\sigma}|\,.$$

Consider the dual chain complex 
$(\Cb^{\bullet}(\widetilde{M}\times X; \mathbb{R}),\delta^{\bullet})$ and the subcomplex of $\Gamma$-invariant chains $(\Cb^{\bullet}(\widetilde{M}\times X; \mathbb{R})^{\Gamma},\delta^{\bullet})$.
Our next goal is to prove an isometric isomorphism similar to the one of Equation \eqref{equation:isom:iso0}, but we first need to pass to the \emph{completion}.

\subsection{The completion.}\label{sec:completion}
Consider the \emph{completion}
$\widehat{\mathcal{C}}_{k}(\Gamma\backslash(\widetilde{M}\times X);\mathbb{R})$ and $\widehat{\mathcal{C}}_{k}(\widetilde{M}\times X;\mathbb{R})$ respectively of the spaces $\mathcal{C}_k(\Gamma\backslash(\widetilde{M}\times X);\mathbb{R})$ and $\mathcal{C}_k(\widetilde{M}\times X;\mathbb{R})$ with respect to the associated norms.
As usual, we quotient out the chains of zero norm, and we therefore obtain the Banach spaces $\widehat{\Cm}_{k}(\Gamma\backslash(\widetilde{M}\times X);\mathbb{R})$
and $\widehat{\Cm}_{k}(\widetilde{M}\times X;\mathbb{R})$.

By \cite[Proposition 1.7]{loh:thesis}, the inclusion $\Ck(\Gamma \backslash(\widetilde{M}\times X);\mathbb{R})\hookrightarrow \widehat{\Cm}_{k}(\Gamma\backslash(\widetilde{M}\times X);\mathbb{R})$ induces isometric maps in homology
$$\iota_k:\Hk(\Gamma\backslash(\widetilde{M}\times X);\mathbb{R})\to \Hk(\widehat{\Cm}_{\bullet}(\Gamma\backslash(\widetilde{M}\times X);\mathbb{R}))\,.$$
In particular, by Theorem \ref{theorem:simplicial:volume} we can compute the simplicial volume via chains in $\widehat{\Cm}_{n}(\Gamma\backslash(\widetilde{M}\times X);\mathbb{R})$, namely
\begin{equation}\label{equation:simp:volume:l1}
  ||M||=|[\Gamma\backslash(\widetilde{M}\times X)]|=|\iota_n([\Gamma\backslash(\widetilde{M}\times X)])|\,.
\end{equation}
Moreover, since any normed chain complex is dense in its completion, passing to the topological dual one immediately has
\begin{equation}\label{eq:isom:bc1}
  \Hb^k(\Gamma \backslash(\widetilde{M}\times X); \mathbb{R})=\Hm^k(\widehat{\Cm}_{\bullet}(\Gamma\backslash(\widetilde{M}\times X);\mathbb{R})')
\end{equation} and 
\begin{equation}\label{eq:isom:bc2}
  \Hm^k(\Cb^{\bullet}(\widetilde{M}\times X; \mathbb{R})^{\Gamma})=\Hm^k((\widehat{\Cm}_{\bullet}(\widetilde{M}\times X;\mathbb{R})')^{\Gamma})\,.
\end{equation}

We now exhibit a natural isometric identification 
$$\widehat{\Cm}_{\bullet}(\Gamma\backslash(\widetilde{M}\times X);\mathbb{R})'\cong (\widehat{\Cm}_{\bullet}(\widetilde{M}\times X;\mathbb{R})')^{\Gamma}\,,$$
between the complexes appearing in the right hand sides of Equation \eqref{eq:isom:bc1} and \eqref{eq:isom:bc2}.
This will allow to compute singular foliated bounded cohomology via bounded functions on non-equivariant chains. 

\begin{lemma}\label{lemma:isom}
  The complexes $(\widehat{\Cm}_{\bullet}(\Gamma\backslash(\widetilde{M}\times X);\mathbb{R})',\delta^\bullet)$ and $((\widehat{\Cm}_{\bullet}(\widetilde{M}\times X;\mathbb{R})')^\Gamma,\delta^\bullet)$ are isometrically isomorphic. 
\end{lemma}
\begin{proof}
We construct mutually inverse cochain maps of norm one. 
First of all, fix $\eta:\Gamma\backslash \Sing_{\bullet}(\widetilde{M})\times X\to\Sing_{\bullet}(\widetilde{M})\times X$ any Borel section of the projection $\pi: \Sing_{\bullet}(\widetilde{M})\times X\to\Gamma\backslash \Sing_{\bullet}(\widetilde{M})\times X$. Given $F\in (\widehat{\Cm}_{\bullet}(\widetilde{M}\times X;\mathbb{R})')^{\Gamma}$ we define 
$f:\Sing_\bullet(\widetilde{M})\times X\to \mathbb{R}$ as follows. For every singular foliated simplex $\sigma:A\to \Gamma\backslash \Sing_\bullet(\widetilde{M})\times X$ set 
$f(\sigma)\coloneqq F(\eta\circ \sigma)$. Then we extend first by linearity to $\text{C}_{\bullet}(\widetilde{M}\times X;\mathbb{R})$ and then to its completion, obtaining a bounded linear function $f:\widehat{\Cm}_{\bullet}(\Gamma\backslash(\widetilde{M}\times X);\mathbb{R})\to \mathbb{R}$. 
Since $\sigma$ fulfils conditions (i), (ii), (iii) of Definition \ref{definition:singular:simplex}, then $\sigma$ satisfies (i), (ii) and (iii) of Definition \ref{definition:singular:simplex:non:eq}. 
We now claim that the above construction does not depend on the section $\eta$. 
To prove this, we pick a different section $\widetilde{\eta}:\Gamma\backslash \Sing_{\bullet}(\widetilde{M})\times X\to\Sing_{\bullet}(\widetilde{M})\times X$. Then, for every $[s,x]\in \Gamma\backslash \Sing_{\bullet}(\widetilde{M})\times X$ the elements 
$\eta([s,x])$ and $\widetilde{\eta}([s,x])$ belong to the same $\Gamma$-orbit in $\Sing_{\bullet}(\widetilde{M})\times X$. In other words, there is a map 
$c_\eta^{\widetilde{\eta}}: \Gamma\backslash \Sing_{\bullet}(\widetilde{M})\times X\to \Gamma$ uniquely determined by the relation
$$c_\eta^{\widetilde{\eta}}([s,x]) \eta([s,x])=\widetilde{\eta}([s,x])\,.$$
The fact that both $\eta$ and $\widetilde{\eta}$ are Borel implies that $c_\eta^{\widetilde{\eta}}$ is Borel.
Consider now the countable Borel partition 
$$\Ima(\sigma) = \bigsqcup\limits_{\gamma\in \Gamma} W_{\gamma}$$ where $W_{\gamma}\coloneqq (c_\eta^{\widetilde{\eta}} )^{-1}(\gamma)\cap \Ima(\sigma)$. Set $A_{\gamma}\coloneqq \sigma^{-1}(W_\gamma)$. Then, by definition, we have 
$\sigma=\sum\limits_{\gamma\in \Gamma}\sigma|_{A_\gamma}$ in $\widehat{\Cm}_{\bullet}(\widetilde{M}\times X; \mathbb{R})$. 
Moreover, for any $\gamma\in \Gamma$ it holds 
$$\widetilde{\eta}\circ \sigma|_{A_\gamma}=\gamma (\eta\circ \sigma|_{A_\gamma})\,.$$
Then, exploiting both the linearity of $F$ and its $\Gamma$-invariance we obtain 
\begin{align*}
  F(\widetilde{\eta}\circ \sigma)&= F\left(\sum\limits_{\gamma\in \Gamma}\widetilde{\eta}\circ \sigma|_{A_\gamma} \right)\\
  &= F\left(\sum\limits_{\gamma\in \Gamma}\gamma (\eta\circ \sigma|_{A_\gamma}) \right)\\
  &= \sum\limits_{\gamma\in \Gamma}F(\gamma(\eta\circ \sigma|_{A_\gamma})) \\
  &=\sum\limits_{\gamma\in \Gamma}F(\eta\circ \sigma|_{A_\gamma}) =F(\eta \circ \sigma)\,,
\end{align*}
which proves the claim.

Conversely, given $f\in \widehat{\Cm}_{\bullet}(\Gamma\backslash(\widetilde{M}\times X);\mathbb{R})'$ we associate (the unique extension of) $F:\textup{C}_{\bullet}(\widetilde{M}\times X;\mathbb{R})\to \mathbb{R}$ where 
$F(\rho)\coloneqq f(\pi\circ \rho)$. 
The fact that $\overline{\rho}=\pi\circ \rho$ lies in $\text{C}_{\bullet}(\Gamma\backslash(\widetilde{M}\times X);\mathbb{R})$ is a consequence of the following considerations where, without loss of generalities, we can assume that $\rho$ is a non-equivariant simplex $A\to \Sing_{\bullet}(\widetilde{M})\times X$. Hence $\overline{\rho}$ is also a simplex. 
\begin{itemize}
  \item[(i)] Since both $\rho$ and $\pi$ are countable-to-one, $\overline{\rho}$ is so. Moreover, a countable union of admissible sets is admissible \cite[Remark 3.3]{sauer}. Thus
  $$\pi^{-1}(\Ima(\overline{\rho}))=\bigsqcup\limits_{\gamma\in \Gamma} \gamma\Ima (\rho)$$
  is admissible. Thus, by definition, $\Ima (\overline{\rho})$ is admissible. 
  \item[(ii)] By condition (ii) of Definition \ref{definition:singular:simplex:non:eq} we get 
  \begin{align*}
  m_{\mathcal{L}}(A)&=\int_{\Ima(\rho)} |\rho^{-1}(s,x)|d\nu(s,x)\\
  &=\sum\limits_{\gamma\in \Gamma} \int_{\Ima(\rho)\cap \mathcal{D}_\bullet} |\rho^{-1}(s,x)|\; d(\gamma_*\nu|_{\mathcal{D}_\bullet})(s,x)\\
  &=\int_{\Ima(\rho)\cap \mathcal{D}_\bullet}\left( \sum\limits_{\gamma\in \Gamma}  |\rho^{-1}(\gamma s,\gamma x)|\right)d\nu|_{\mathcal{D}_\bullet}(s,x)\\
  &=\int_{\Ima(\overline{\rho})} |\overline{\rho}^{-1}([s,x])|d\overline{\nu}([s,x])\,. 
  \end{align*} 
  In the above computation we used the decomposition of Example \ref{example:decomposition} and the definition of $\overline{\nu}$ via the bijection $\pi|_{\mathcal{D}_\bullet}: \mathcal{D}_\bullet\to \Gamma\backslash\Sing_{\bullet}(\widetilde{M})\times X$.
\item [(iii)] For every $x\in X $ we have 
$$\Ima(\pi\circ \rho)\cap \mathcal{L}_\Sigma(x)= \bigsqcup\limits_{\gamma\in \Gamma} \gamma^{-1} \left( \Ima(\rho)\cap (\widetilde{M})\times \{\gamma x\} \right)$$
which is locally finite by condition (iii) of Definition \ref{definition:singular:simplex:non:eq}
\end{itemize}
The fact the the constructions are mutually inverse is a straightforward verification. Moreover, they both commute with the coboundary operators and do not affect the norm, hence they define mutually inverse isometric cochain maps. 
This concludes the proof.
\end{proof}

The identification of Lemma \ref{lemma:isom} descends in cohomology to isometric isomorphisms
\begin{equation}\label{eq:isom:iso}
  \Hm^k((\widehat{\Cm}_{\bullet}(\widetilde{M}\times X;\mathbb{R})')^\Gamma)\cong \Hb^k(\widetilde{M}\times X; \mathbb{R})\,.
\end{equation}
This, together with the equalities \eqref{eq:isom:bc1} and \eqref{eq:isom:bc2}, is the starting point to prove the main result of this section.
\begin{rec_thm}[\ref{theorem:bounded:cohomology}]
  Let $M$ be a triangulated oriented closed connected aspherical manifold. Let $p: \widetilde{M}\to M$ be its universal cover and consider an essentially free measure preserving action $\Gamma=\pi_1(M)\curvearrowright (X,\mu)$ on a standard Borel probability space. Then we have canonical isometric isomorphisms 
  $$\Hb^k(\Gamma \backslash(\widetilde{M}\times X); \mathbb{R})\cong \Hb^k(\Gamma;\Linf(X,\mathbb{R}))\cong \Hmb^k(\Gamma\ltimes X;\mathbb{R})$$
  for every $k\geq 0$. 
\end{rec_thm}
\begin{proof}
  The strategy is to show that the cochain complex $(\Cm_{\bullet}(\widetilde{M}\times X;\mathbb{R})',\delta^{\bullet})$ provides a strong resolution by relative injective $\Gamma$-modules of the $\Gamma$-module $\Linf(X,\mathbb{R})$. Thanks to the Fundamental Lemma of Homological Algebra Lemma \ref{lemma:fundamental}, to the equality of Equation \eqref{eq:isom:bc2} and to the isometric isomorphisms of Equation \eqref{eq:isom:iso} this allows to conclude that $\Hb^k(\Gamma \backslash(\widetilde{M}\times X); \mathbb{R})$ is canonically isomorphic to $\Hb^n(\Gamma;\Linf(X,\mathbb{R}))$. 
Then, for the isometricity part, we exhibit a norm non-increasing cochain $\Gamma$-map $\alpha^{\bullet}: \Cb^{\bullet}(\Gamma;\Linf(X,\mathbb{R}))\to \Cm_{\bullet}(\widetilde{M}\times X;\mathbb{R})'$ extending the identity of $\Linf(X,\mathbb{R})$. Thus we apply again Lemma \ref{lemma:fundamental} to conclude.
The second isomorphism is precisely the one of Equation \eqref{equation:exponential:isomorphism}.

First of all, in order to turn the cochain complex $(\Cm_{\bullet}(\widetilde{M}\times X;\mathbb{R})',\delta^{\bullet})$ into a resolution of $\Linf(X,\mathbb{R})$, we need to define an augmentation map $\epsilon: \Linf(X,\mathbb{R})\to \Cm_{0}(\widetilde{M}\times X;\mathbb{R})'$.
Given a simplex $\sigma\in S_0(\widetilde{M}\times X)$ we denote by $C\subset \Sing_0(\widetilde{M})$ the countable set such that $\Ima(\sigma)\subset C\times X$. For every $F\in \Linf(X,\mathbb{R})$ we pick a representative $f$ and set
\begin{equation}\label{eq:augmentation}
  \epsilon (F)(\sigma)\coloneqq \int_X\sum\limits_{s\in C} f(x) |\sigma^{-1}
  (s, x)| d\mu(x)=\int_{\Ima(\sigma)} f(x) \omega_{\sigma}(s,x) d\nu(y,x)\,.
\end{equation}
Then we extend to $\mathcal{C}_{0}(\widetilde{M}\times X;\mathbb{R})$ by linearity. Notice that the aboe quantity depends on the values of the function $f$ only up to null-sets, hence a different choice of representative for $F$ does not change its value. This shows that $\epsilon$ is well-defined. By Equation \eqref{eq:augmentation} it is immediate that $\epsilon$ vanishes on $\mathcal{C}_{0}^{\text{a.e.}}(\widetilde{M}\times X;\mathbb{R})$ and that
$$||\epsilon||_{\infty}=\sup_{||f||_{\infty}\leq 1} ||\epsilon(f)||_{\infty}\leq 1\,.$$
Hence $\epsilon$ induces a bounded map, still denoted by $\epsilon$, on $\textup{C}_0(\widetilde{M}\times X;\mathbb{R})$. 
We check that $\delta^0\circ \epsilon=0$. Consider $d_{\pm}:\Sing_1(\widetilde{M})\to \Sing_0(\widetilde{M})$ the face operators so that $\partial_0$ is defined as the post-composition by $d_+\otimes \id_X -d_-\otimes \id_X$. We pick $\sigma\in S_1(\widetilde{M}\times X;\mathbb{R})$ 
with $\Ima(\sigma)\subset C \times X$ and set $$C_+=\{d_+(s)\,,\,s\in C\}\,,\quad C_-=\{d_-(s)\,,\,s\in C\}\,.$$ Since for every $s\in \Sing_1(\widetilde{M})$ it holds $\omega_{\sigma}(s,x)=\omega_{\partial_0 \sigma}(d_+ s,x)=\omega_{\partial_0 \sigma}(d_- s,x)$ we have
\begin{align*}
  \delta^0\circ \epsilon(F)(\sigma)&=\epsilon(F)(\partial_0\sigma)\\ 
  &=\int_X\left(\sum\limits_{s\in C_+} f(x) \omega_{\partial_0\sigma}(s,x) -\sum\limits_{s\in C_-} f(x) \omega_{\partial_0  \sigma}(s,x)\right) d\mu(x)\\
  &=\int_X\sum\limits_{s\in C}  (f(x)  \omega_{\sigma}(s,x) -f(x)  \omega_{\sigma}(s,x)) \nu(s,x)=0
\end{align*}
Moreover, $\epsilon$ is $\Gamma$-equivariant, indeed
\begin{align*}
  \epsilon(\gamma\cdot F)(\sigma)&=\int_X\sum\limits_{s\in C} (\gamma \cdot f)(x) |\sigma^{-1}
  (s, x)| d\mu(x)\\
&=\int_X\sum\limits_{s\in \gamma^{-1} C}  f(\gamma^{-1}x) |(\gamma^{-1}\cdot \sigma)^{-1}
  ( s, \gamma^{-1} x)| d\mu(x)\\
  &=\int_X\sum\limits_{s\in \gamma^{-1} C}  f(x) |(\gamma^{-1}\cdot \sigma)^{-1}
  ( s,  x)| d\gamma_*\mu(x)\\
  &=\int_X\sum\limits_{s\in \gamma^{-1} C}  f(x) |(\gamma^{-1}\cdot \sigma)^{-1}
  ( s,  x)| d\mu(x)\\
  &= \epsilon(F)(\gamma^{-1}\cdot\sigma)= \gamma\cdot\epsilon(F ) (\sigma)\,.
\end{align*}
Thus the augmented complex 
$$0\to \Linf(X,\mathbb{R})\xrightarrow{\epsilon}\Cm_{0}(\widetilde{M}\times X;\mathbb{R})' \xrightarrow{\delta^0}\Cm_{1}(\widetilde{M}\times X;\mathbb{R})'\to \cdots $$
is a resolution of $\Linf(X,\mathbb{R})$ by $\Gamma$-modules. 

We claim that each module is relatively injective.
To prove this fact we consider the following extension 
problem 
\begin{center}
  \begin{tikzcd}
E \arrow{rr}{\iota}\arrow[swap]{rd}{\psi} && W\arrow[bend right =20,swap]{ll}{\varphi}\arrow[dotted]{ld}{?}\\
& \Cm_{k}(\widetilde{M}\times X;\mathbb{R})'\,.&
  \end{tikzcd}    
\end{center}
Given a fundamental domain $D_k$ of the action $\Gamma\curvearrowright \Sing_k(\widetilde{M})$, we denote by  $\pi_{D_k}:\Sing_k(\widetilde{M})\times X\to \Gamma$ the map that associates to each pair $(s,x)$ the unique $\gamma\in \Gamma$ such that $\gamma^{-1}s \in D_k$. Given $\sigma\in S_k(\widetilde{M}\times X)$, we set $\sigma_{D_k}:A\to \Gamma$ as $\sigma_{D_k}\coloneqq \pi_{D_k}\circ \sigma$. Then we define 
$$\beta: W\rightarrow \textup{C}_{k}(\widetilde{M}\times X;\mathbb{R})'\,,\;\;\; 
\beta(w) (\sigma)\coloneqq \frac{1}{|\sigma|} \int_A \psi (\sigma_{D_k}(t)\cdot \varphi(\sigma_{D_k}(t)^{-1} \cdot w))(\sigma) dm_{\mathcal{L}}(t)\,,$$
for every $w\in W$ and $\sigma:A\to \Sing_k(\widetilde{M})\times X$.
To see that $\beta$ is a solution of the extension problem, we first compute
\begin{align*}
  \beta(\iota(w))(\sigma)
  &=\frac{1}{|\sigma|} \int_A \psi (\sigma_{D_k}(t)\cdot \varphi(\sigma_{D_k}(t)^{-1} \cdot \iota(w)))(\sigma) dm_{\mathcal{L}}(t)\\
  &=\frac{1}{|\sigma|} \int_A \psi (\sigma_{D_k}(t)\cdot \varphi \circ \iota (\sigma_{D_k}(t)^{-1} \cdot w))(\sigma) dm_{\mathcal{L}}(t)\\
  &=\frac{1}{|\sigma|} \int_A \psi (\sigma_{D_k}(t) \sigma_{D_k}(t)^{-1} \cdot w)(\sigma) dm_{\mathcal{L}}(t)\\
  &=\frac{1}{|\sigma|} \int_A \psi (w) (\sigma) dm_{\mathcal{L}}(t)\\
  &=\frac{m_{\mathcal{L}}(A)}{|\sigma|} \psi (w) (\sigma)=\psi (w) (\sigma)\,,
\end{align*}
where we to pass from the first to the second line we exploited the $\Gamma$-invariance of $\iota$, then we used the relation $\varphi\circ \iota=\id$ and we concluded since, by definition, $|\sigma|=m_{\mathcal{L}}(A)$. 

Secondly, for any $\gamma\in \Gamma$, we have
\begin{align*}
  \beta (\gamma\cdot w)(\sigma)
  &=\frac{1}{|\sigma|} \int_A \psi (\sigma_{D_k}(t)\cdot \varphi(\sigma_{D_k}(t)^{-1} \cdot (\gamma \cdot w)))(\sigma) dm_{\mathcal{L}}(t)\\
  &=\frac{1}{|\sigma|} \int_A \psi (\sigma_{D_k}(t)\cdot \varphi((\gamma^{-1} \sigma_{D_k}(t))^{-1}  \cdot w))(\sigma) dm_{\mathcal{L}}(t)\\
  &=\frac{1}{|\gamma^{-1}\cdot \sigma|} \int_A \psi (\gamma^{-1} \sigma_{D_k}(t)\cdot \varphi((\gamma^{-1} \sigma_{D_k}(t))^{-1}  \cdot w))(\gamma^{-1} \cdot \sigma) dm_{\mathcal{L}}(t)\\
  &=\beta ( w)(\gamma^{-1} \cdot \sigma)=\gamma\cdot \beta ( w)(\sigma)\,.
\end{align*}
In the above computation we first applied the $\Gamma$-equivariance of $\psi$ and then the fact that $(\gamma\cdot\sigma)_{D_k}(t)=\gamma\cdot\sigma_{D_k}(t)$. 

Finally, the inequality $||\beta||_{\infty}\leq||\psi||_{\infty}$ follows since $\varphi$ is norm non-increasing and because the operator 
$$S_k(\widetilde{M}\times X;\mathbb{R})\to \mathbb{R}\,,\;\;\; \sigma\mapsto 
\frac{1}{|\sigma|} \int_A \psi(v)(\sigma) dm_{\mathcal{L}}(t)$$
has norm at most $||\psi||_{\infty}$.
This concludes the proof about the relative injectivity. 

We move on our proof by showing that the complex $(\Cm_{\bullet}(\widetilde{M}\times X;\mathbb{R})',\delta^{\bullet})$ is strong. In fact, a contracting homotopy can be constructed as follows.
Fix a point $\widetilde{y}_0\in\widetilde{M}$. 
We denote by $T_k:\Sing_k(\widetilde{M})\to\Sing_{k+1}(\widetilde{M})$ the map described in \cite[Lemma 5.4]{miolibro} having the following property: for every $s\in \Sing_k(\widetilde{M})$, the $0$-th
vertex of $T_k(s)$ is equal to $\widetilde{y}_0$, and has $s$ as opposite face.
We set $T_k^X:\Sing_k(\widetilde{M})\times X\to\Sing_{k+1}(\widetilde{M})\times X$ as $T_k^X=T_k^X(s,x)=(T_k(s),x)$. 
For every $k\geq 1$ and any simplex $\sigma:A\to \Sing_k(\widetilde{M})\times X$ define
$$\mathcal{T}_k(\sigma):A\to \Sing_{k+1}(\widetilde{M})\times X\,,\;\;\; \mathcal{T}_k(\sigma)\coloneqq 
T_k^X\circ \sigma\,. $$
By duality we set
$$\theta^k: \Cm_{k}(\widetilde{M}\times X;\mathbb{R})'\to\Cm_{k-1}(\widetilde{M}\times X;\mathbb{R})'$$
as (the extension by density of) the functional
$$\theta^k(f)(\sigma)\coloneqq f(\mathcal{T}_k(\sigma))\,,$$
Moreover, we set
$$\theta^0: \Cm_{0}(\widetilde{M}\times X;\mathbb{R})'\to\Linf(X,\mathbb{R})$$ as the class of 
$$\theta^0(f)(x)\coloneqq f(\sigma_{\widetilde{y}_0})\,,$$
where $\sigma_{\widetilde{y}_0}:X\to \widetilde{M}\times X$ is defined as
$\sigma_{\widetilde{y}_0}(x)=(\widetilde{y}_0,x)$. 
A direct application of the definition of both $\delta^{\bullet}$ and $\theta^{\bullet}$ implies that the following equalities
$$\theta^{k+1}\circ \delta^k+\delta^{k-1}\circ \theta^k=\id_{\Cm_{k}(\widetilde{M}\times X;\mathbb{R})'}\,,\;\;\; 
\theta^{1}\circ \delta^0+\epsilon\circ \theta^0=\id_{\Cm_{0}(\widetilde{M}\times X;\mathbb{R})'}$$
hold. 
This shows that $\theta$ is a contracting homotopy. In particular, thanks to the isometric isomorphism of Equation \eqref{eq:isom:iso}, we can conclude that there exist canonical isomorphisms 
$$\Hb^k(\Gamma \backslash(\widetilde{M}\times X); \mathbb{R})\cong \Hb^k(M;\Linf(X,\mathbb{R}))\,$$ 
in every degrees. 

As anticipated, in order to prove isometricity we need to find a norm non-increasing $\Gamma$-equivariant cochain map 
$$\iota^{\bullet}:\Cb^{\bullet}(\Gamma;\Linf(X,\mathbb{R}))\to \Cm_{\bullet}(\widetilde{M}\times X;\mathbb{R})'$$
extending the identity on $\Linf(X,\mathbb{R})$. Let $D$ be a fundamental domain for the action $\Gamma\curvearrowright \widetilde{M}$.
We consider the map $r: \Sing_0(\widetilde{M})\to \Gamma$ such that $r(s)^{-1} s(*)\in D_k$ where $*$ denotes the standard $0$-simplex. 
For a singular simplex $s\in \Sing_n(\widetilde{M})$ we denote by $s_i=s\circ v_i$ the $i$-th vertex of $s$, namely the $0$-simplex obtained precomposing $s$ with the embedding $$v_i:\Delta_0\hookrightarrow \Delta_n\,,\quad v_i(*)\coloneqq e_i\,.$$ 
We also see an element $F\in \Cb^{n}(\Gamma;\Linf(X,\mathbb{R}))$ as a class in $\Linf(\Gamma^{n+1}\times X, \mathbb{R})$ via the exponential isomorphism \eqref{equation:exponential:isomorphism}. We set
$$\iota^n(F)(\sigma)\coloneqq \frac{1}{|\sigma|} \int\limits_{X} \sum_{s\in C} f(r(s_0),\ldots,r(s_n))(x)|\sigma^{-1}(s,x)| d\mu(x)\,$$
for every representative $f$ of $F$ and every $\sigma\in S_n(\widetilde{M}\times X;\mathbb{R})$ with $\Ima(\sigma)\subset C\times X$. 
As above, another representative would differ from $f$ on a full measure subset of $X$, whence the right hand side would not change.
We now show that $\iota^{\bullet}$ is a cochain map extending the identity.
Consider $F\in \Linf(X,\mathbb{R})$, a representative $f$ and $\sigma\in S_0(\widetilde{M}\times X)$. Then we if $\varepsilon:\Linf(X,\mathbb{R})\to \Cb^0(\Gamma;\Linf(X,\mathbb{R}))$ denotes the inclusion of coefficients, we have 
\begin{align*}
  \iota^0(\varepsilon(F))(\sigma)
  &= \frac{1}{|\sigma|} \int\limits_{X} \sum_{s\in C} f(x)|\sigma^{-1}(s,x)| d\mu(x)=\epsilon(F)(\sigma)\,.
\end{align*}
For $n\geq 1$, if $\sigma\in S_{n+1}(\widetilde{M}\times X)$ with $\Ima(\sigma)\subset C\times X$, we set 
$C_i\coloneqq \{d_i(s)\,,\, s\in C\}\,.$ Then for any $F\in \Cb^n(\Gamma;\Linf(X,\mathbb{R}))$ we pick a representative $f$ and compute
\begin{align*}
  \iota^{n+1}(\delta_n(f))(\sigma)&= \frac{1}{|\sigma|} \int\limits_{X} \sum_{s\in C} \delta_n(f)(r(s_0),\ldots,r(s_{n+1}))(x)|\sigma^{-1}(s,x)| d\mu(x)\\
  &= \frac{1}{|\sigma|} \int\limits_{X} \sum_{s\in C} \sum\limits_{i=0}^{n+1} (-1)^{i} f(r(s_0),\ldots,\widehat{r(s_i)},\ldots,r(s_{n+1}))(x)|\sigma^{-1}(s,x)| d\mu(x)\\
  &= \frac{1}{|\sigma|} \int\limits_{X} \sum_{s\in C} \sum\limits_{i=0}^{n+1}  (-1)^{i} f(r(d_i(s)_0),\ldots,r(d_i(s)_{n}))(x)|\sigma^{-1}(d_i(s),x)| d\mu(x)\\
  &= \sum\limits_{i=0}^{n+1}  (-1)^{i}  \frac{1}{|\partial^i_{n+1}\sigma|} \int\limits_{X} \sum_{s\in C_i} f(r(s_0),\ldots,r(s_{n}))(x)|(\partial^i_{n+1}\sigma)^{-1}(s,x)| d\mu(x)\\
   &= \sum\limits_{i=0}^{n+1}  (-1)^{i}\iota^n(F)(\partial^i_{n+1}\sigma)=\delta^n( \iota^n(F))(\sigma)\,.
\end{align*}
Above we used the decomposition of $\partial_{n+1}$ of Equation \eqref{eq:boundary} and the equality
$$|\sigma^{-1}(d_i(s),x)|=|(\partial_{n+1}^i\sigma)^{-1} (s,x)|$$
that holds for every $s\in C_i$. As before, this computation does not depend on the choice of the representative $f$ of $F$.

To show the $\Gamma$-equivariance we pick any $\gamma\in \Gamma$ and we compute 
\begin{align*}
  \iota^n(\gamma\cdot F)(\sigma)&= \frac{1}{|\sigma|} \int\limits_{X}\sum_{s\in C} (\gamma \cdot f)(r(s_0),\ldots,r(s_n))(x)|\sigma^{-1}(s,x)| d\mu(x)\\
  &= \frac{1}{|\sigma|} \int\limits_{X}\sum_{s\in C} f(\gamma^{-1}r(s_0),\ldots,\gamma^{-1}r(s_n))(\gamma^{-1}x)|\sigma^{-1}(s,x)|  d\mu(x)\\
  &= \frac{1}{|\gamma^{-1}\cdot \sigma|} \int\limits_{X}\sum_{s\in \gamma C}  f(r(s_0),\ldots,r(s_n))(x)|(\gamma^{-1}\sigma)^{-1}(s,x)| d\mu(x)\\
  &=\iota^n(F)(\gamma^{-1}\cdot\sigma)=\gamma\cdot \iota^n(F)(\sigma)\,.
\end{align*}
In the above computation we moved from the second line to the third one exploiting the $\Gamma$-invariance of both $\mu$ and the norm on singular foliated chains. 

Moreover, since by definition $$\int\limits_X\sum_{s\in C}
|\sigma^{-1}(s,x)| d\mu(x)=m_{\mathcal{L}}(A)=|\sigma|$$
we have that $||\iota^n(F)||\leq||F||$ for every $F$, which implies that $\iota^n$ does not increase the norm. This concludes the proof. 
\end{proof}

Denote by $\langle\cdot,\cdot\rangle$ the Kronecker pairing between 
$\Hm_n(\Gamma\backslash(\widetilde{M}\times X);\mathbb{R})$ and $\Hb^n(\Gamma\backslash(\widetilde{M}\times X);\mathbb{R})$.
The duality principle of Lemma \ref{lemma:duality} implies the following consequence of Theorem \ref{theorem:simplicial:volume}.
\begin{cor}\label{corollary:duality}
  Let $M$ be a triangulated oriented closed connected aspherical $n$-manifold. Let $p: \widetilde{M}\to M$ be its universal cover and consider an essentially free measure preserving action $\Gamma=\pi_1(M)\curvearrowright (X,\mu)$ on a standard Borel probability space. Then 
  $$||M||=\max \{\langle[\Gamma\backslash(\widetilde{M}\times X)],\lambda\rangle|\,|\, \lambda\in \Hb^n(\Gamma\backslash(\widetilde{M}\times X);\mathbb{R})\,,\,||\lambda||_{\infty}\leq 1 \} \,.$$
\end{cor}

We immediately deduce the following vanishing criterion.

\begin{rec_cor}[\ref{corollary:vanishing}]
    Let $M$ be a triangulated oriented closed connected aspherical $n$-manifold. Let $p: \widetilde{M}\to M$ be its universal cover and consider an essentially free measure preserving action $\Gamma=\pi_1(M)\curvearrowright (X,\mu)$ on a standard Borel probability space. If $\Hmb^n(\Gamma\ltimes X;\mathbb{R})$ (equivalently $\Hb^n(\Gamma;\Linf(X,\mathbb{R}))\cong \Hb^n(M;\Linf(X,\mathbb{R}))$) vanishes, then also $||M||$ does. 
\end{rec_cor}
\begin{proof}
  This follows by Corollary \ref{corollary:duality} and by Theorem \ref{theorem:bounded:cohomology}.
\end{proof}

\begin{oss}
  In the above setting, assume that the action $\Gamma\curvearrowright (X,\mu)$ is \emph{amenable} in the sense of Zimmer \cite[Definition 4.3.1]{zimmer}. A special case of \cite[Theorem 4]{sarti:savini:groupoids} (see also \cite[Proposition 7.4.1]{monod:libro}) shows that $\Hmb^k(\Gamma\ltimes X;\mathbb{R})=0$ for every $k\geq 1$. 
  In particular, by Corollary \ref{corollary:vanishing} we have $||M||=0$. One may wonder whether this criterion provides new examples of manifolds with vanishing simplicial volume. In fact, this is not the case, since by \cite[Proposition 4.3.3]{zimmer} any amenable p.m.p action is given by an amenable group. Whence we fall into the classic vanishing results for manifolds with amenable fundamental group. In other words, interesting examples satisfying the hypothesis of Corollary \ref{corollary:duality} should be investigated outside the amenable world. 
\end{oss}

We conclude with a further application of Theorem \ref{theorem:bounded:cohomology} to \emph{transverse measured groupoids}.
\begin{defn}\label{def:cross}
  Let $\Gamma\curvearrowright (X,\mu)$ be a p.m.p action of a countable group on a standard Borel probability space.
  A \emph{cross-section} is a Borel subset $Y\subset X$ such that $Y$ meets every $\Gamma$-orbit. 
\end{defn}
We recall that such sections are known to exist even in the more general setting of measured groupoids (e.g. \cite[Theorem 2.8]{FHM} or \cite[Theorem 5.6]{ramsayJFA}).
It should be pointed out that, in general, $Y$ needs not to have positive measure in $X$. For example, for a transitive group action $\Gamma\curvearrowright \Gamma/\Gamma_0$ where $\Gamma_0<\Gamma$ is a closed subgroup (e.g. a lattice) a cross section is given by the singleton. For this reason, it makes no sense to consider the restriction groupoid with the measure induced by $\Gamma\ltimes X$ \cite[Section 2]{sarti:savini:groupoids}. Nevertheless, by classical results due to Connes and others, one can build a measure $\tau$ on $Y$, called \emph{transverse measure}, starting from the one on $X$ \cite{BHK}. Equipped with such measured structure, the restriction $\Gamma\ltimes X|_Y$ is called \emph{transverse measured groupoid} and carries some important information about the initial action. For instance, their measurable bounded cohomology coincide \cite{HS}. We refer to the fundational papers \cite{sarti:savini:groupoids,sarti:savini25} for the reader interested in the general theory of measurable bounded cohomology of measured groupoids. From the above considerations we deduce the following  
\begin{rec_cor}[\ref{corollary:transverse}]
  Let $M$ be a triangulated oriented closed connected aspherical $n$-manifold. Let $p: \widetilde{M}\to M$ be its universal cover and consider an essentially free measure preserving action $\Gamma=\pi_1(M)\curvearrowright (X,\mu)$ on a standard Borel probability space. Let $Y\subset X$ be any cross-section. If $\Hmb^n(\Gamma\ltimes X|_Y;\mathbb{R})$ vanishes, then also $||M||$ does. 
\end{rec_cor}
\begin{proof}
This follows by Corollary \ref{corollary:vanishing} and by the isomorphism $\Hmb^n(\Gamma\ltimes X;\mathbb{R})\cong \Hmb^n(\Gamma\ltimes X|_Y;\mathbb{R})$ \cite[Theorem 1]{HS}.
\end{proof}

  The above result is a potential source of new examples of manifolds with vanishing simplicial volume. Indeed, it would suffice finding a p.m.p action of the fundamental group and a transverse measured groupoid with vanishing bounded cohomology. 
  For instance, a (non-necessarily amenable) p.m.p action with a cross-section whose associated transverse measured groupoid is amenable would have this property \cite[Theorem 4]{sarti:savini:groupoids}. 
 However, no non-trivial examples of this scenario are currently known. In fact, the absence of explicit computations in bounded cohomology is even more pronounced in the context of groupoids, where no advanced computational techniques have yet been developed.

\bibliographystyle{alpha}

\bibliography{biblionote}

\end{document}